\documentclass[a4paper,leqno]{amsart}

\usepackage{latexsym}
\usepackage[english]{babel}
\usepackage{fancyhdr}
\usepackage[mathscr]{eucal}
\usepackage{amsmath}
\usepackage{mathrsfs}
\usepackage{amsthm}
\usepackage{amsfonts}
\usepackage{amssymb}
\usepackage{amscd}
\usepackage{bbm}
\usepackage{graphicx}
\usepackage{graphics}
\usepackage{latexsym}
\usepackage{color}

\newcommand{\ud}{\mathrm{d}}

\newcommand{\ii}{\mathrm{i}}

\newcommand{\ran}{\mathrm{ran}}

\theoremstyle{plain}
\newtheorem{theorem}{Theorem}[section]
\newtheorem{lemma}[theorem]{Lemma}

\newtheorem{proposition}[theorem]{Proposition}

\theoremstyle{definition}

\newtheorem{remark}[theorem]{Remark}
\newtheorem*{remark*}{Remark}

\numberwithin{equation}{section}

\begin{document}

\title[Fractional powers, singular perturbations, quantum Hamiltonians]
{Fractional powers and singular perturbations of quantum differential Hamiltonians}
\author[A.~Michelangeli]{Alessandro Michelangeli}
\address[A.~Michelangeli]{International School for Advanced Studies -- SISSA \\ via Bonomea 265 \\ 34136 Trieste (Italy).}
\email{alemiche@sissa.it}
\author[A.~Ottolini]{Andrea Ottolini}
\address[A.~Ottolini]{Department of Mathematics, Stanford University \\ 450 Serra Mall, Stanford CA 94305 (USA).}
\email{ottolini@stanford.edu}
\author[R.~Scandone]{Raffaele Scandone}
\address[R.~Scandone]{International School for Advanced Studies -- SISSA \\ via Bonomea 265 \\ 34136 Trieste (Italy).}
\email{rscandone@sissa.it}

\dedicatory{Dedicated to Gianfausto Dell'Antonio on the occasion of his 85th birthday}

\begin{abstract}
We consider the fractional powers of singular (point-like) perturbations of the Laplacian, and the singular perturbations of fractional powers of the Laplacian, and we compare such two constructions focusing on their perturbative structure for resolvents and on the local singularity structure of their domains. In application to the linear and non-linear Schr\"{o}dinger equations for the corresponding operators we outline a programme of relevant questions that deserve being investigated.
\end{abstract}

\date{\today}

\subjclass[2000]{}
\keywords{Fractional Quantum Mechanics, fractional Laplacian, singular perturbations of differential operators, Kre{\u\i}n-Vi\v{s}ik-Birman self-adjoint extension theory.
}

%

\maketitle


\section{Background: at the edge of fractional quantum mechanics and zero-range interactions}

At the edge of the theory of quantum Hamiltonians with zero-range interactions, the theory of partial differential operators, and the recent mainstream of fractional quantum mechanics, there are two constructions, each of which is classical in nature, the combination of which has been receiving an increasing attention in the recent times.


The first is the construction of fractional powers of a differential operator with non-negative symbol and more generally the fractional power of a non-negative self-adjoint operator on $L^2(\mathbb{R}^d)$. For concreteness, let us focus on the negative Laplacian
\[
 -\Delta\;=\;-\sum_{j=1}^d\frac{\partial^2}{\partial x_j^2}\,,\qquad x\equiv(x_1,\dots,x_d)\in\mathbb{R}^d\,,
\]
that we shall simply call the Laplacian. In this case the definition of $(-\Delta)^{s/2}$ for $s\in\mathbb{R}$ is obvious in terms of the corresponding power of the Fourier multiplier for $-\Delta$: 
\[
 ((-\Delta)^{s/2}f)(x)\;=\;(|p|^s\widehat{f})^\vee(x)\,,
\]
where $(\mathcal{F}f)(p)\equiv\widehat{f}(p)=(2\pi)^{-\frac{d}{2}}\int_{\mathbb{R}^d}e^{-\ii x p}f(x)\ud x$ is the present convention for the Fourier transform. In fact, $\mathcal{F}(-\Delta)\mathcal{F}^*$ gives an explicit multiplication form and hence an explicit spectral decomposition for $-\Delta$ as a non-negative self-adjoint operator on $L^2(\mathbb{R}^d)$, thus $(-\Delta)^{s/2}$ given by the identity above coincides with the construction with functional calculus.

The second construction is that kind of perturbation of a given pseudo-differential operator that heuristically amounts to add to it a potential supported at one point only, whence the jargon of singular perturbation \cite{albeverio-solvable,albverio-kurasov-2000-sing_pert_diff_ops}. This is a typical restriction-extension construction, where first one restricts the initial operator to sufficiently smooth functions vanishing in neighbourhoods of a given point $x_0\in\mathbb{R}^d$, and then one builds an operator extension of such restriction that is self-adjoint on $L^2(\mathbb{R}^d)$. This procedure, when the initial operator is $-\Delta$, is known to be equivalent to the somewhat more concrete scheme of taking the limit $\varepsilon\downarrow 0$ in Schr\"{o}dinger operators of the form $-\Delta+V_\varepsilon$, where $V_\varepsilon$ is a regular potential essentially supported at a scale $\varepsilon^{-1}$ around $x_0$ and magnitude diverging with $\varepsilon$, which shrink to a delta-like profile centred at $x_0$ \cite{albeverio-solvable,albverio-kurasov-2000-sing_pert_diff_ops}.

Both constructions are well known and relevant in several contexts: Sobolev spaces, fractional Sobolev norms, high and low regularity theory for non-linear PDE, etc., concerning the former; solvable yet realistic models with computable spectral features (eigenvalues, scattering matrix,...)~in quantum mechanics, chemistry, biology, acoustics, concerning the latter.

Recently, especially for the solution theory of non-linear Schr\"{o}dinger equations whose linear part is governed by singular Hamiltonians of point interactions\cite{Georgiev-M-Scandone-2016-2017,MOS-SingularHartree-2017}, as well as for linear Schr\"{o}dinger-like equations with singular perturbations of fractional powers of the Laplacian \cite{Muslih-IntJTP-2010-3D-fracLaplandDelta,COliveira-Costa-Vaz-JMP2010,COliveira-Vaz-JPA2011_1D_fracLaplandDelta,Lenzi-etAl-JMP2013_fracLapl_andDelta,Sandev-Petreska-Lenzi-JMP2014,Tare-Esguerra-JMP2014,Jarosz-Vaz-JMP2016_1D_gs_fracLaplandDelta,Nayga-Esguerra-JMP-2016,Sacchetti-fractional2018}, the interest has increased around various ways of \emph{combining the two constructions above}.

The goal of this work is to set up the problem in a systematic way for two operations that in a sense do not commute, to make the rigorous constructions of the operators of interest, and to make qualitative and quantitative comparisons.

We shall consider, in the appropriate sense that we are going to specify, the class of singular perturbations of the $d$-dimensional Laplacian, supported at a point $x_0\in\mathbb{R}^d$, in the homogeneous and inhomogeneous case, namely (informally speaking),
\begin{equation}\label{eq:symb1}
 \begin{split}
  \mathfrak{h}_\tau\;&=\;-\Delta+\textrm{singular perturbation at $x_0$} \\
  \mathfrak{h}_\tau+\mathbbm{1}\;&=\;-\Delta+\mathbbm{1}+\textrm{singular perturbation at $x_0$}\,,
 \end{split}
\end{equation}
and then we shall combine this construction with that of fractional powers, thus considering on the one hand the operators
\begin{equation}\label{eq:symb2b}
 \begin{split}
  \mathfrak{h}_\tau^{s/2}\;&=\;(-\Delta+\textrm{singular perturbation at $x_0$})^{s/2} \\
  (\mathfrak{h}_\tau+\mathbbm{1})^{s/2}\;&=\;(-\Delta+\mathbbm{1}+\textrm{singular perturbation at $x_0$})^{s/2}
 \end{split}
\end{equation}
and on the other hand the operators
\begin{equation}\label{eq:symb2}
 \begin{split}
  \mathsf{k}_\tau^{(s/2)}\;&=\;(-\Delta)^{s/2}+\textrm{singular perturbation at $x_0$} \\
  \mathfrak{d}
  _\tau^{(s/2)}\;&=\;(-\Delta+\mathbbm{1})^{s/2}+\textrm{singular perturbation at $x_0$}\,.
 \end{split}
\end{equation}
All operators in \eqref{eq:symb1}--\eqref{eq:symb2} above are meant to be self-adjoint.

Let us remark that $\mathfrak{h}_\tau^{s/2}$ and $(\mathfrak{h}_\tau+\mathbbm{1})^{s/2}$ are going to be genuine fractional powers of a non-negative self-adjoint operator on $L^2(\mathbb{R}^3)$ (to this aim one has to consider non-restrictively only those singular perturbations that produce non-negative operators), whereas  the different notation for the superscript $s/2$ in $\mathsf{k}_\tau^{(s/2)}$ and in $\mathfrak{d}_\tau^{(s/2)}$ is to indicate that the latter operators are instead singular perturbations of $s/2$-th powers (and not fractional powers of singular perturbations).

In all cases $\tau\in(\mathbb{R}\cup\{\infty\})^{\mathcal{J}^2}$, for some $\mathcal{J}\in\mathbb{N}$, is going to be a parameter that qualifies one element in the infinite family of self-adjoint realisations of the considered singular perturbation, with the customary convention that `$\tau=\infty$' denotes the absence of perturbation.

As for the choice of the point $x_0$, there is no loss of generality in choosing $x_0=0$, which we shall do henceforth.

The knowledge of singular quantum Hamiltonians of the type \eqref{eq:symb1} is well established in the literature and we review them in Section \ref{sec:halpha_and_pert}. The study of their fractional powers, hence of the operators of the type \eqref{eq:symb2b}, has only started recently and in the second part of Section \ref{sec:halpha_and_pert} we give an account of the main known facts about them.

In comparison to \eqref{eq:symb1} and \eqref{eq:symb2b}, we shall then discuss the rigorous construction of operators of the type \eqref{eq:symb2} (Sections \ref{sec:homo-general}--\ref{sec:high_power}).
Our presentation will have two main focuses, which reflect into the formulation of our results:
\begin{itemize}
 \item[1.] to qualify the nature of the perturbation in the resolvent sense (\emph{finite rank} vs \emph{infinite-rank} perturbations);
 \item[2.] to qualify the natural decomposition of the domain of the considered operators into a \emph{regular component} and a \emph{singular component}, and to determine the boundary condition constraining such two components.
\end{itemize}
The first issue is central for deducing an amount of properties from the unperturbed to the perturbed operators. The second issue also arises naturally, as one can see heuristically that the considered operators must act in an ordinary way on those functions supported away from the perturbation centre, and therefore their domains must contain a subspace of $H^s$-regular functions, where $s$ is the considered power, next to a more singular component that is the signature of the perturbation.

In this respect we are going to highlight profound differences between two constructions, say,
\[
 \begin{split}
  (-\Delta+\mathbbm{1}+\textrm{singular perturbation at $x_0$})^{s/2} \\
  (-\Delta+\mathbbm{1})^{s/2}+\textrm{singular perturbation at $x_0$}\,,
 \end{split}
\]
that are therefore `non-commutative'.

We organized the material as follows:
\begin{itemize}
 \item the analysis of operators of fractional power of point interaction Hamiltonians is presented in Section \ref{sec:halpha_and_pert};
 \item on the other hand, the construction of singular perturbations of fractional Laplacians is presented in Sections \ref{sec:homo-general} through \ref{sec:high_power}, with the general set-up in Section \ref{sec:homo-general}, the detailed discussion of the paradigmatic scenario of rank-one perturbations in Sections \ref{sec:homog} and \ref{sec:homog-1} (homogeneous case) and in Section \ref{sec:inhomog} (inhomogeneous case), and an outlook on high-rank perturbations in Section \ref{sec:high_power};
 \item last, in Section \ref{sec:perspectives} we outline an amount of relevant questions that deserve being investigated in application to the linear and non-linear Schr\"{o}dinger equations governed by the operators constructed in this work.
\end{itemize}

\section{Singular perturbations of the Laplacian and their fractional powers}\label{sec:halpha_and_pert}

In this Section we qualify the operators $\mathfrak{h}_\tau$, $\mathfrak{h}_\tau^{s/2}$, and $(\mathfrak{h}_\tau+\mathbbm{1})^{s/2}$ of \eqref{eq:symb1}-\eqref{eq:symb2b}. 
Whereas the former is well known in all dimensions in which it is not trivial, namely $d=1,2,3$, the latter two operators have been studied mainly in three dimensions, thus in this Section we choose $d=3$ for our presentation.

For chosen $\lambda>0$ we set
\begin{equation}
  G_\lambda(x)\;:=\;\frac{e^{-\sqrt{\lambda}|x|}}{4\pi |x|}\;=\;\frac{1}{\:(2\pi)^{\frac{3}{2}}}\Big(\frac{1}{p^2+\lambda}\Big)^{\!\vee}(x)\,,\qquad x,p\in\mathbb{R}^3\,,
\end{equation}
whence
\begin{equation}
 (-\Delta+\lambda)G_\lambda\;=\;\delta(x)
\end{equation}
as a distributional identity on $\mathbb{R}^3$. We also set
\begin{equation}
 \mathring{\mathfrak{h}}\;:=\;\overline{-\Delta\upharpoonright C^\infty_0(\mathbb{R}^3\!\setminus\!\{0\})}
\end{equation}
as an operator closure with respect to the Hilbert space $L^2(\mathbb{R}^3)$.

As well known\cite{MO-2016}, $\mathring{\mathfrak{h}}$ is densely defined, closed, and symmetric on $L^2(\mathbb{R}^3)$, with
 \begin{equation}\label{eq:domhring}
  \mathcal{D}(\mathring{\mathfrak{h}})\;=\;\Big\{f\in H^2(\mathbb{R}^3)\,\Big|\!\int_{\mathbb{R}^3}\widehat{f}(p)\,\ud p=0\Big\}\,,\qquad \mathring{\mathfrak{h}} f=-\Delta f\,.
 \end{equation}
 Its Friedrichs extension given by the self-adjoint Laplacian on $L^2(\mathbb{R}^3)$ with domain $H^2(\mathbb{R}^3)$, and its adjoint is the operator
  \begin{equation}\label{eq:decompositionD_Hdostar}
  \begin{split}
   \mathcal{D}(\mathring{\mathfrak{h}}^*)\;&=\;\left\{g\in L^2(\mathbb{R}^3)\left|\!
  \begin{array}{c}
  \widehat{g}(p)=\displaystyle\widehat{f^\lambda}(p)+\frac{\eta}{(p^2+\lambda)^2}+\frac{\xi}{p^2+\lambda} \\
  f^\lambda\in\mathcal{D}(\mathring{\mathfrak{h}})\,,\quad \eta,\xi\in\mathbb{C}
  \end{array}\!\!\!\right.\right\} \\
  &=\;H^2(\mathbb{R}^3)\dotplus\mathrm{span}\{G_\lambda\} \\
    (\widehat{(\mathring{\mathfrak{h}}^*+\lambda) g)}\,(p)\;&=\; (p^2+\lambda) \,\Big(\widehat{f^\lambda}(p)+\frac{\eta}{(p^2+\lambda)^2}\Big)\,.
  \end{split}
 \end{equation}

The structure in Eq.~\eqref{eq:decompositionD_Hdostar} is typical of a well-known decomposition (see, e.g., Eq.~(2.5) and (2.10) in \cite{GMO-KVB2017}). We observe that $\mathcal{D}(\mathring{\mathfrak{h}}^*)$ does \emph{not} depend on $\lambda$, only the splitting of $g\in\mathcal{D}(\mathring{\mathfrak{h}}^*)$ into a $H^2$-function plus a less regular component does. The last identity in \eqref{eq:decompositionD_Hdostar} may be re-interpreted distributionally as
\[
 \mathring{\mathfrak{h}}^* g\;=\;-\Delta g + (2\pi)^{\frac{3}{2}}\xi\,\delta(x) g\,,
\]
where neither $-\Delta g$ nor $(2\pi)^{\frac{3}{2}}\xi\,\delta(x) g$ belongs to $L^2(\mathbb{R}^3)$, but their sum does.

The fact that  $\ker(\mathring{\mathfrak{h}}^*+\lambda\mathbbm{1})=\mathrm{span}\{G_\lambda\}$ indicates that $\mathring{\mathfrak{h}}$ has deficiency index 1 and hence admits a one-(real-)parameter family of self-adjoint extensions. They can be classified in terms of the standard parametrisation of the Kre{\u\i}n-Vi\v{s}ik-Birman self-adjoint extension theory (see, e.g.,\cite[Sec.~3]{GMO-KVB2017}), identifying each of them as a restriction of $\mathring{\mathfrak{h}}^*$ by imposing in \eqref{eq:decompositionD_Hdostar} Birman's self-adjointness condition $\eta=\tau\xi$ for some $\tau\in\mathbb{R}\cup\{\infty\}$ (see, e.g., \cite[Theorem 1 and Corollary 1]{MO-2016}), whence the following Theorem.

\begin{theorem}\label{thm:defDeltaalpha}
 Let $\lambda>0$. 
 \begin{itemize}
  \item[(i)] The self-adjoint extensions in $L^2(\mathbb{R}^3)$ of the operator $\mathring{\mathfrak{h}}$ form the family $(\mathfrak{h}_\tau)_{\tau\in\mathbb{R}\cup\{\infty\}}$, where $\mathfrak{h}_\infty=\mathring{\mathfrak{h}}_F$, the Friedrichs extension, and all other (proper) extensions are given by $\mathfrak{h}_\tau=\mathring{\mathfrak{h}}^*\upharpoonright\mathcal{D}(\mathfrak{h}_\tau)$, where
  \begin{equation}\label{eq:D_HringTau}
  \begin{split}
   \mathcal{D}(\mathfrak{h}_\tau)\;:=&\;\left\{g\in L^2(\mathbb{R}^3)\!\left|\!
  \begin{array}{c}
  \widehat{g}(p)=\displaystyle\widehat{f^\lambda}(p)+\frac{\tau\,\xi}{(p^2+\lambda)^2}+\frac{\xi}{p^2+\lambda} \\
  \xi\in\mathbb{C}\,,\quad f^\lambda\in\mathcal{D}(\mathring{\mathfrak{h}}) 
  \end{array}\!\!\!\right.\right\} \\
  =&\;\;\Big\{\,g=F^\lambda+{\textstyle\frac{8\pi\sqrt{\lambda}}{\tau}}\,F^\lambda(0)\,G_\lambda\,\Big|\,F^\lambda\in H^2(\mathbb{R}^3) \Big\}\,.
  \end{split}
 \end{equation}
 \item[(ii)] Each extension is semi-bounded from below and
  \begin{equation}\label{eq:positiveHring-tau_iff_positve_tau}
 \begin{split}
 \inf\sigma(\mathfrak{h}_\tau+\lambda\mathbbm{1})\;\geqslant \;0\quad&\Leftrightarrow\quad \tau\;\geqslant\; 0 \\
 \inf\sigma(\mathfrak{h}_\tau+\lambda\mathbbm{1})\;> \;0\quad&\Leftrightarrow\quad \tau\;>\; 0 \\
 (\mathfrak{h}_\tau+\lambda\mathbbm{1})\textrm{ is invertible}\quad&\Leftrightarrow\quad \tau\neq 0\,.
 \end{split}
 \end{equation}
 \item[(iii)] For  each $\tau\in\mathbb{R}$ the quadratic form of the extension $\mathfrak{h}_\tau$ is given by
\begin{eqnarray}
   \mathcal{D}[\mathfrak{h}_\tau]\;&=&\; H^1(\mathbb{R}^3)\dotplus \mathrm{span}\{G_\lambda\}  \label{eq:Hring-tau_form1}\\
 \qquad \mathfrak{h}_\tau[F^\lambda+\kappa_\lambda G_\lambda]\;&=&\; \|\nabla F^\lambda\|_{L^2(\mathbb{R}^3)}^2-\lambda\|F^\lambda+\kappa_\lambda G_\lambda\|_{L^2(\mathbb{R}^3)}^2+\lambda\|F^\lambda\|_{L^2(\mathbb{R}^3)}^2 \nonumber \\
 & & \quad +\frac{\tau}{8\sqrt{\lambda}}|\kappa_\lambda|^2 \label{eq:Hring-tau_form2}
 \end{eqnarray}
 for any $F^\lambda\in H^1(\mathbb{R}^3)$ and $\kappa_\lambda\in\mathbb{C}$.
 \end{itemize}
\end{theorem}

\begin{remark}\label{rem:taulambda}
 The $\tau$-parametrisation depends on the initial choice of the $\lambda$-shift and thus the same extension is described by infinitely many different pairs $(\lambda,\tau)$ -- of course with a unique $\tau$ once $\lambda$ is chosen. This is clear by inspecting the boundary condition at $x=0$ between regular and singular component of a generic $g\in  \mathcal{D}(\mathfrak{h}_\tau)$: for any two pairs $(\lambda,\tau)$ and $(\lambda',\tau')$ identifying the \emph{same} extension $\mathfrak{h}_\tau$, owing to \eqref{eq:D_HringTau} one has
 \[
  \begin{split}
   \widehat{g}\;=&\;\widehat{F^\lambda}+\frac{\xi}{p^2+\lambda}\;=\;\widehat{F^{\lambda'}}+\frac{\xi'}{p^2+\lambda'} \\
   \textrm{with}\quad\;\xi'\;:=&\;{\textstyle\frac{8\pi\sqrt{\lambda}}{\tau'\,(2\pi)^3}\big(\!\int_{\mathbb{R}^3}\widehat{F^\lambda}\ud p\big)}\,,\quad \widehat{F^{\lambda'}}\;:=\;\widehat{F^\lambda}+{\textstyle\frac{\xi'(\lambda'-\lambda)}{(p^2+\lambda)(p^2+\lambda')}}\,,
  \end{split}
 \]
 and also
 \[
  \begin{split}
    \textstyle\int_{\mathbb{R}^3}\widehat{F^\lambda}\ud p\;&=\; \textstyle\frac{\tau\,(2\pi)^3}{8\pi\sqrt{\lambda}}\,\xi \\
    \textstyle\int_{\mathbb{R}^3}\widehat{F^{\lambda'}}\ud p\;&=\; \textstyle\frac{\tau'\,(2\pi)^3}{8\pi\sqrt{\lambda'}}\,\xi'\cdot\frac{\tau\sqrt{\lambda'}}{\tau'\sqrt{\lambda}}\big(1+\frac{2\sqrt{\lambda}(\sqrt{\lambda'}-\sqrt{\lambda})}{\tau}\big)\,,
  \end{split}
 \]
 whence necessarily $1=\frac{\tau\sqrt{\lambda'}}{\tau'\sqrt{\lambda}}\big(1+\frac{2\sqrt{\lambda}(\sqrt{\lambda'}-\sqrt{\lambda})}{\tau}\big)$, or equivalently
 \begin{equation}\label{eq:taulambdatauprimelprime}
  \frac{\tau-2\lambda}{\sqrt{\lambda}}\;=\;\frac{\tau'-2\lambda'}{\sqrt{\lambda'}}\,.
 \end{equation}
 Thus, when referring to the extension $\mathfrak{h}_\tau$, we shall always implicitly declare the choice of $\lambda$, and any other $\lambda'$ and $\tau'$ satisfying \eqref{eq:taulambdatauprimelprime} actually identify the same extension.
\end{remark}

In the literature it is customary to find the second expression of \eqref{eq:D_HringTau} with the alternative extension parameter $\alpha$ defined by
\begin{equation}\label{eq:alphatau}
 \alpha\;:=\;\frac{\tau-2\lambda}{8\pi\sqrt{\lambda}}\,,
\end{equation}
thus re-writing $(\mathfrak{h}_\tau)_{\tau\in\mathbb{R}\cup\{\infty\}}$ as $(\mathfrak{h}_\alpha)_{\alpha\in\mathbb{R}\cup\{\infty\}}$ with $\mathfrak{h}_\alpha:=\mathring{\mathfrak{h}}^*\upharpoonright\mathcal{D}(\mathfrak{h}_\alpha)$ (see \cite[Remark 3]{MO-2016} and \cite[Eq.~(I.1.1.26)]{albeverio-solvable}). Of course, owing to Remark \ref{rem:taulambda}, in particular to formula \eqref{eq:taulambdatauprimelprime}, the parameter $\alpha$ identifies unambiguously an extension irrespectively of the shift $\lambda$ chosen for the explicit domain decomposition. From this and a bit of further spectral analysis \cite[Sec.~I.1.1]{albeverio-solvable} one then deduces the following.

\begin{theorem}~

\begin{itemize}
 \item[(i)] The self-adjoint extensions in $L^2(\mathbb{R}^3)$ of the operator $\mathring{\mathfrak{h}}$ form the family $(\mathfrak{h}_\alpha)_{\alpha\in\mathbb{R}\cup\{\infty\}}$, where $\mathfrak{h}_\infty=\mathring{\mathfrak{h}}_F$, the Friedrichs extension, and all other (proper) extensions are given by $\mathfrak{h}_\alpha:=\mathring{\mathfrak{h}}^*\upharpoonright\mathcal{D}(\mathfrak{h}_\alpha)$, where, for arbitrary $\lambda>0$
 \begin{equation}\label{eq:domwithalpha}
 \begin{split}
  \mathcal{D}(\mathfrak{h}_\alpha)\;&=\;\Big\{\,g=F^\lambda+{\textstyle(\alpha+\frac{\sqrt{\lambda}}{4\pi})^{-1}}F^\lambda(0)\,G_\lambda\,\Big|\,F^\lambda\in H^2(\mathbb{R}^3) \Big\} \\
  (\mathfrak{h}_\alpha+\lambda)\,g\;&=\;(-\Delta+\lambda)\,F^\lambda\,.
 \end{split}
\end{equation}
 \item[(ii)] The spectrum of $\mathfrak{h}_\alpha$ is given by
\begin{equation}\label{eq:spectrumDalpha}
\begin{split}
 \sigma_{\mathrm{ess}}(\mathfrak{h}_\alpha)\;&=\;\sigma_{\mathrm{ac}}(\mathfrak{h}_\alpha)\;=\;[0,+\infty)\,,\qquad \sigma_{\mathrm{sc}}(\mathfrak{h}_\alpha)\;=\;\emptyset\,, \\
 \sigma_{\mathrm{p}}(\mathfrak{h}_\alpha)\;&=\;
\begin{cases}
\qquad \emptyset & \textrm{if }\alpha\in[0,+\infty] \\
\{-(4\pi\alpha)^2\} & \textrm{if }\alpha\in(-\infty,0)\,.
\end{cases}
\end{split}
\end{equation}
The negative eigenvalue $-(4\pi\alpha)^2$, when it exists, is non-degenerate and the corresponding eigenfunction is $|x|^{-1}e^{-4\pi|\alpha|\,|x|}$. Thus, $\alpha\geqslant 0$ corresponds to a non-confining, `repulsive' contact interaction.
 \item[(iii)] For  each $\alpha\in\mathbb{R}$ the quadratic form of the extension $\mathfrak{h}_\alpha$ is given by
  \begin{equation}\label{eq:form_dom-opaction}
 \begin{split}
  \mathcal{D}[\mathfrak{h}_\alpha]\;&=\;H^1(\mathbb{R}^3)\dotplus\mathrm{span}\{ G_\lambda\} \\
  \mathfrak{h}_\alpha[F^\lambda+\kappa_\lambda\,G_\lambda] \;&=\; -\lambda\|F^\lambda+\kappa_\lambda\,G_\lambda\|_{L^2(\mathbb{R}^3)}^2 \\
  &\qquad+\|\nabla F^\lambda\|_{L^2(\mathbb{R}^3)}^2+\lambda\|F^\lambda\|_{L^2(\mathbb{R}^3)}^2+{\textstyle \big(\alpha+\frac{\sqrt{\lambda}\,}{4\pi}\big)}\,|\kappa_\lambda|^2
 \end{split}
\end{equation}
 for arbitrary $\lambda>0$.
  \item[(iv)] The resolvent of $\mathfrak{h}_\alpha$ is given by
  \begin{equation}\label{eq:resolvhalpha}
   (\mathfrak{h}_\alpha+\lambda\mathbbm{1})^{-1}\;=\;(-\Delta++\lambda\mathbbm{1})^{-1}+{\textstyle(\alpha+\frac{\sqrt{\lambda}}{4\pi})^{-1}} |G_\lambda\rangle\langle G_\lambda|
  \end{equation}
  for arbitrary $\lambda>0$.
\end{itemize}
\end{theorem}

For the operators  $\mathfrak{h}_\alpha+\lambda\mathbbm{1}$ with $\lambda$ sufficiently large, and the operators $\mathfrak{h}_\alpha$ with $\alpha\geqslant 0$, the self-adjoint functional calculus defines the powers $(\mathfrak{h}_\alpha+\lambda\mathbbm{1})^{s/2}$ or $\mathfrak{h}_\alpha^{s/2}$ for $s\in\mathbb{R}$.

A surely relevant regime is surely $s\in[0,2]$: the integer powers $s=2,1,0$ correspond, respectively, to the considered operator, its square root whose domain is then the form domain of the considered operator, and the identity; the fractional powers in between are of interest when one needs to discuss the corresponding linear or non-linear dynamics in spaces of convenient fractional regularity.

From the thorough analysis of such fractional powers made in \cite{Georgiev-M-Scandone-2016-2017}, one has the following.

\begin{theorem}\label{thm:fracpowGMS} Let $\alpha\geqslant 0$, $\lambda>0$, and $s\in[0,2]$, and set 
 \begin{equation}
  \begin{split}
   \widetilde{H}^s_\alpha(\mathbb{R}^3)\;&:=\;\mathcal{D}(\mathfrak{h}_\alpha^{s/2})\;=\;\mathcal{D}(\mathfrak{h}_\alpha+\lambda\mathbbm{1})^{s/2} \\
   \|g\|_{\widetilde{H}^s_\alpha(\mathbb{R}^3)}\;&:=\;\|(\mathfrak{h}_\alpha+\lambda\mathbbm{1})^{s/2}g\|_{L^2(\mathbb{R}^3)}\,.
  \end{split}
\end{equation}
 \begin{itemize}
  \item[(i)] One has
  \begin{equation}\label{eq:Ds_s0-1/2}
 \begin{split}
   \quad\;\,\qquad\qquad\widetilde{H}^s_\alpha(\mathbb{R}^3)\;&=\;H^s(\mathbb{R}^3) \\
   \|\psi\|_{\widetilde{H}^s_\alpha(\mathbb{R}^3)}\;&\approx\;\|\psi\|_{H^s(\mathbb{R}^3)} \qquad\qquad\qquad\qquad\qquad\textstyle\textrm{if $s\in[0,\frac{1}{2})$}\,,
 \end{split}
 \end{equation}
  \begin{equation}\label{eq:Ds_s1/2-3/2}
 \begin{split}
   \widetilde{H}^s_\alpha(\mathbb{R}^3)\;&=\;H^s(\mathbb{R}^3)\dotplus\mathrm{span}\{G_\lambda\}\\
   \|F^\lambda+\kappa_\lambda\,G_\lambda\|_{\widetilde{H}^s_\alpha(\mathbb{R}^3)}\;&\approx\;\|F^\lambda\|_{H^s(\mathbb{R}^3)}+(1+\alpha)|\kappa_\lambda|\qquad\quad\;\;\:\textstyle\textrm{if $s\in(\frac{1}{2},\frac{3}{2})$}\,,
 \end{split}
 \end{equation}
  \begin{equation}\label{eq:Ds_s3/2-2}
 \begin{split}
   \widetilde{H}^s_\alpha(\mathbb{R}^3)\;&=\;\Big\{g=F^\lambda+
   \big({\textstyle \alpha+\frac{\sqrt{\lambda}}{4\pi}\big)^{-1}} F^\lambda(0)
   \,G_\lambda\,\Big|\, F^\lambda\in H^s(\mathbb{R}^3)\Big\}\\
   \|F^\lambda&+\kappa_\lambda\,G_\lambda\|_{\widetilde{H}^s_\alpha(\mathbb{R}^3)}\;\approx\;\|F^\lambda\|_{H^s(\mathbb{R}^3)}\,\:\:\quad\qquad\qquad\qquad\textstyle\textrm{if $s\in(\frac{3}{2},2)$}\,,
 \end{split}
 \end{equation}
where here `$\approx$' denotes the equivalence of norms.
  \item[(ii)] For $g\in\mathcal{D}(\mathfrak{h}_\alpha^{s/2})$,
\begin{equation}\label{eq:fract_pow_formula}
\begin{split}
 (\mathfrak{h}_\alpha+\lambda\mathbbm{1})^{s/2} g\;&= \;(-\Delta+\lambda\mathbbm{1})^{s/2} g \\
 & \quad -\,4\sin{\textstyle\frac{s \pi}{2}}\int_0^{+\infty}\!\ud t\; \frac{t^{s/2}\,\langle G_{\lambda+t},g\rangle_{L^2(\mathbb{R}^3)}}{\,4\pi\alpha+\sqrt{\lambda+t}\,}\,G_{\lambda+t}\,.
\end{split}
\end{equation}
  \item[(iii)] One has the resolvent identity
  \begin{equation}\label{eq:fract_res}
 (\mathfrak{h}_\alpha+\lambda\mathbbm{1})^{-s/2}\;= \;(-\Delta+\lambda\mathbbm{1})^{-s/2}-\!\int_0^{+\infty}\!\!\ud t\; {\textstyle\frac{\,4t^{-\frac{s}{2}} \sin{\textstyle\frac{s \pi}{2}}}{\,4\pi\alpha+\sqrt{\lambda+t}\,}}\,|G_{\lambda+t}\rangle\langle G_{\lambda+t}|\,.
\end{equation}
 \end{itemize}
\end{theorem}

 The transition cases $s=\frac{1}{2}$ and $s=\frac{3}{2}$ too can be qualified, however  with less explicit formulas -- see \cite[Prop.~8.1 and 8.2]{Georgiev-M-Scandone-2016-2017}.
 
 Thus, as described in Theorem \ref{thm:fracpowGMS}, for $s>\frac{1}{2}$, $\widetilde{H}^s_\alpha(\mathbb{R}^3)$ decomposes into a regular $H^s$-component and a singular component with local singularity $|x|^{-1}$, \emph{precisely as the domain of $\mathfrak{h}_\alpha$ itself}, and for $s>\frac{3}{2}$ regular and singular parts are constrained by a local boundary condition among them \emph{of the same type as in} \eqref{eq:domwithalpha}; low powers $s<\frac{1}{2}$, instead, only produce domains where no leading singularity can be singled out.
 
 It is remarkable that one has such an explicit and clean knowledge of the fractional powers of the singular perturbations of the Laplacian: the singular perturbation yields an operator $\mathfrak{h}_\alpha$ that is not a (pseudo-)differential operator any longer and its powers are not simply recovered as Fourier multipliers, nor is it a priori obvious how the fractional power affects the local boundary condition between regular and singular components of elements in $\mathcal{D}(\mathfrak{h}_\alpha)$. 
 
 It is also worth remarking that whereas $\mathfrak{h}_\alpha$, $\mathfrak{h}_\alpha^{s/2}$, and $(\mathfrak{h}_\alpha+\lambda\mathbbm{1})^{s/2}$ for $s\in(\frac{1}{2},2)$ share the same singular behaviour $|x|^{-1}$ of the functions in their domain, one noticeable difference is given by equations \eqref{eq:resolvhalpha} and \eqref{eq:fract_res}: indeed, $(\mathfrak{h}_\alpha+\lambda\mathbbm{1})^{-1}$ is a rank-one perturbation of $(-\Delta+\lambda\mathbbm{1})^{-1}$, while $(\mathfrak{h}_\alpha+\lambda\mathbbm{1})^{-s/2}$ is not a finite rank perturbation of $(-\Delta+\lambda\mathbbm{1})^{-s/2}$.

\section{Singular perturbations of the fractional Laplacian: general setting for the homogeneous case}\label{sec:homo-general}

In comparison to $\mathfrak{h}_\tau^{s/2}$, the fractional power of a singular perturbation of the Laplacian, we start discussing in this Section the rigorous construction of operators of the type $\mathsf{k}_\tau^{(s/2)}$, as introduced informally in \eqref{eq:symb2}, namely the self-adjoint singular perturbations of the fractional power of the Laplacian. Then, in Section \ref{sec:homog} we shall consider the concrete cases of dimension $d=3$, and in Section \ref{sec:homog-1} the case $d=1$.


For chosen $d\in\mathbb{N}$, $\lambda>0$, and $s\in\mathbb{R}$ we set
\begin{equation}\label{eq:sfG_generic}
  \mathsf{G}_{s,\lambda}(x)\;:=\;\frac{1}{\:(2\pi)^{\frac{d}{2}}}\Big(\frac{1}{|p|^{s}+\lambda}\Big)^{\!\vee}(x)\,,\qquad x,p\in\mathbb{R}^d\,,
\end{equation}
whence
\begin{equation}
 ((-\Delta)^{s/2}+\lambda) \,\mathsf{G}_{s,\lambda}\;=\;\delta(x)
\end{equation}
as a distributional identity on $\mathbb{R}^d$.
We also set
\begin{equation}\label{eq:dom_k_closure}
 \mathring{\mathsf{k}}^{(s/2)}\;:=\;\overline{(-\Delta)^{s/2}\upharpoonright C^\infty_0(\mathbb{R}^d\!\setminus\!\{0\})}
\end{equation}
as an operator closure with respect to the Hilbert space $L^2(\mathbb{R}^d)$. 
Thus, in comparison to Section \ref{sec:halpha_and_pert}, when $d=3$ one has $\mathsf{G}_{2,\lambda}=G_\lambda$ and $\mathring{\mathsf{k}}^{(1)}=\mathring{\mathfrak{h}}$.

The domain of $\mathring{\mathsf{k}}^{(s/2)}$, as a consequence of the operator closure in \eqref{eq:dom_k_closure}, is a space of functions with $H^s$-regularity and vanishing conditions at $x=0$ for each function and its partial derivatives. The amount of vanishing conditions depends on $d$ and $s$, to classify which we introduce the intervals 
\begin{equation}
  I_{n}^{(d)}\;:=\;
 \begin{cases}
  \qquad(0,\frac{d}{2}) & n=0 \\
  (\frac{d}{2}+n-1,\frac{d}{2}+n) & n =1,2,\dots
 \end{cases} 
 \end{equation}
For our purposes it is convenient to use momentum coordinates to express the vanishing conditions that qualify the domain of $\mathring{\mathsf{k}}^{(s/2)}$: thus, with the notation $p\equiv(p_1,\dots,p_d)\in\mathbb{R}^d$, by means of an approximation argument (see Appendix \ref{closurecharact}) we find
\begin{equation}\label{eq:d-dim-domain}
 \begin{split}
  \mathcal{D}(\mathring{\mathsf{k}}^{(s/2)})\;&=\;H^s_0(\mathbb{R}^d\!\setminus\!\{0\})\;=\;\overline{C^\infty_0(\mathbb{R}^d\!\setminus\!\{0\})\,}^{\|\,\|_{H^s}} \\
  &=\;\begin{cases}
  \qquad\qquad H^s(\mathbb{R}^d) & \textrm{if } s\in I_0^{(d)}\\
  & \\
  \;\left\{\!\!
 \begin{array}{c}
  f\in H^s(\mathbb{R}^d) \textrm{ such that}\\
  \int_{\mathbb{R}^d}p_1^{\gamma_1}\cdots p_d^{\gamma_d}\widehat{f}(p)\,\ud p=0 \\
  \gamma_1,\dots,\gamma_d\in\mathbb{N}_0\,,\;\sum_{j=1}^d\gamma_j\leqslant n-1
 \end{array}\!\!
 \right\} & \textrm{if } s\in I_n^{(d)}\,,\;n =1,2,\dots
 \end{cases}
 \end{split}
\end{equation}
Clearly, $ \int_{\mathbb{R}^d}p_1^{\gamma_1}\cdots p_d^{\gamma_d}\widehat{f}(p)\,\ud p=0$ is the same as $\big(\frac{\partial^{\gamma_1}}{\partial x_1^{\gamma_1}}\cdots \frac{\partial^{\gamma_d}}{\partial x_d^{\gamma_d}}f\big)(0)=0$, with the notation $x\equiv(x_1,\dots,x_d)\in\mathbb{R}^d$.

The expression of $\mathcal{D}(\mathring{\mathsf{k}}^{(s/2)})$ for the endpoint values $s=\frac{d}{2}+n$ require an amount of extra analysis besides the arguments  proof of \eqref{eq:d-dim-domain}: we do not discuss it here, an omission that does not affect the conceptual structure of our presentation.

Being densely defined, closed, and positive, either the symmetric operator $\mathring{\mathsf{k}}^{(s/2)}$ is already self-adjoint on $L^2(\mathbb{R}^d)$, or it admits infinitely many self-adjoint extensions. As well known, the infinite multiplicity of such extensions is quantified by the \emph{deficiency index} of $\mathring{\mathsf{k}}^{(s/2)}$, which is the quantity
\begin{equation}
 \mathcal{J}(s,d)\;:=\;\dim\ker\big((\mathring{\mathsf{k}}^{(s/2)})^*+\lambda\mathbbm{1}\big)
\end{equation}
for one, and hence for all $\lambda>0$. The self-adjointness of $\mathring{\mathsf{k}}^{(s/2)}$ is equivalent to $\mathcal{J}(s,d)=0$.

It is not difficult to compute $\mathcal{J}(s,d)$ for generic values of $d$ and $s$ and to identify a natural basis of the $\mathcal{J}(s,d)$-dimensional space $\ker\big((\mathring{\mathsf{k}}^{(s/2)})^*+\lambda\mathbbm{1}\big)$.

\begin{lemma}\label{lem:rule}
 For given $d\in\mathbb{N}$ and $s>0$, 
 \begin{equation}\label{eq:rule}
 s\in I_{n}^{(d)}\quad \Rightarrow \quad 
 \mathcal{J}(s,d)\;=\;
 \textstyle\begin{pmatrix}
  d+n-1 \\ d
 \end{pmatrix}.
\end{equation}
In particular, when $s\in I_{n}^{(d)}$ for some $n\in\mathbb{N}$, then
\begin{equation}\label{eq:kerstar-generic}
 \ker\big((\mathring{\mathsf{k}}^{(s/2)})^*+\lambda\mathbbm{1}\big)\;=\;\mathrm{span}\Big\{ u^{\lambda}_{\gamma_1,\dots,\gamma_d}\,\Big|\, \gamma_1,\dots,\gamma_d\in\mathbb{N}_0\,,\;\sum_{j=1}^d\gamma_j\leqslant n-1\Big\}\,,
\end{equation}
where
 \begin{equation}\label{eq:uxifunctions}
  \widehat{u^{\lambda}}_{\gamma_1,\dots,\gamma_d}(p)\;:=\;\frac{\,p_1^{\gamma_1}\cdots p_d^{\gamma_d}\,}{\,|p|^s+\lambda\,}\,.
 \end{equation}
\end{lemma}

It is worth noticing, comparing \eqref{eq:sfG_generic} and \eqref{eq:uxifunctions}, that
\begin{equation}\label{eq:uxi-G}
 u^{\lambda}_{0,\dots,0}\;=\;(2\pi)^{\frac{d}{2}}\,\mathsf{G}_{s,\lambda}\,.
\end{equation}

\begin{proof}[Proof of Lemma \ref{lem:rule}]
When $s\in I_0^{(d)}$, we see from \eqref{eq:d-dim-domain} that $\mathring{\mathsf{k}}^{(s/2)}$ is self-adjoint: then $\ker\big((\mathring{\mathsf{k}}^{(s/2)})^*+\lambda\mathbbm{1}\big)$ is trivial and $\mathcal{J}(s,d)=0$, consistently with \eqref{eq:rule}. When $s\in I_n^{(d)}$, $n=1,2,\dots$, then $u\in\ker\big((\mathring{\mathsf{k}}^{(s/2)})^*+\lambda\mathbbm{1}\big)=\ran\big(\mathring{\mathsf{k}}^{(s/2)}+\lambda\mathbbm{1}\big)^\perp$ is equivalent to
 \[
  0\;=\int_{\mathbb{R}^3}\widehat{u}(p)(|p|^s+\lambda)\widehat{f}(p)\ud p\qquad\forall f\in\mathcal{D}(\mathring{\mathsf{k}}^{(s/2)})
 \]
and one argues from \eqref{eq:d-dim-domain} that $\ker\big((\mathring{\mathsf{k}}^{(s/2)})^*+\lambda\mathbbm{1}\big)$ is spanned by linearly independent functions of the form $u^{\lambda}_{\gamma_1,\dots,\gamma_d}$, with $\sum_{j=1}^d\gamma_j\leqslant n-1$. Such functions are as many as the linearly independent monomials in $d$ variables with degree at most equal to $n-1$, and therefore their number equals $\begin{pmatrix} d+n-1 \\ d \end{pmatrix}$.
\end{proof}

The knowledge of $\ker\big((\mathring{\mathsf{k}}^{(s/2)})^*+\lambda\mathbbm{1}\big)$ and of the inverse of the Friedrichs extension of $\mathring{\mathsf{k}}^{(s/2)}$ 
are the two inputs for the Kre{\u\i}n-Vi\v{s}ik-Birman extension theory (see, e.g., \cite[Sec.~3]{GMO-KVB2017}), by means of which we can produce the whole family of self-adjoint extensions of $\mathring{\mathsf{k}}^{(s/2)}$.

Such a construction is particularly clean in the case, relevant in applications, of deficiency index one: the comprehension of this case is instructive to understand the case of higher deficiency index. Moreover, as we shall see, in this case the self-adjoint extensions of $\mathring{\mathsf{k}}^{(s/2)}$ turns out to be rank-one perturbations, in the resolvent sense: we will use the jargon $\mathcal{J}=1$ or `rank one' interchangeably.

Therefore, in this work we make the presentational choice to discuss in detail the $\mathcal{J}(s,d)=1$ scenario when $s\in I_1^{(d)}$, deferring to Section \ref{sec:high_power} an outlook on the high-$\mathcal{J}$ scenario.
This corresponds to analysing the regimes $s\in(\frac{1}{2},\frac{3}{2})$ when $d=1$, $s\in(1,2)$ when $d=2$, $s\in(\frac{3}{2},\frac{5}{2})$ when $d=3$, etc.

The construction of the self-adjoint extensions of $\mathring{\mathsf{k}}^{(s/2)}$ in any such regimes is technically the very same, irrespectively of $d$, except for a noticeable peculiarity when $d=1$, as opposite to $d=2,3,\dots$

Indeed, when $s\in I_1^{(d)}$ and hence $\mathcal{J}(s,d)=1$, we know from Lemma \ref{lem:rule} and \eqref{eq:uxi-G} that $\ker\big((\mathring{\mathsf{k}}^{(s/2)})^*+\lambda\mathbbm{1}\big)=\mathrm{span}\{\mathsf{G}_{s,\lambda}\}$, and the function $\mathsf{G}_{s,\lambda}$ may or may not have a \emph{local singularity} as $x\to 0$. As follows from the $d$-dimensional distributional identity
\[
 2^{\frac{s}{2}}\Gamma({\textstyle\frac{s}{2}})\,\frac{1}{\,|p|^s}\;=\;2^{\frac{d-s}{2}}\Gamma({\textstyle\frac{d-s}{2}}) \,\widehat{\Big(\frac{1}{\,|x|^{d-s}}\Big)}\,,\qquad s\in(0,d)\,,
\]
$\mathsf{G}_{s,0}$ has a singularity $\sim|x|^{-(d-s)}$ when $s<d$, it has a logarithmic singularity when $s=d$, and it is continuous at $x=0$ when $s>d$, with asymptotics
\begin{equation}\label{eq:sfGlambda_asympt}
 \begin{split}
  |x|^{d-s}\,\mathsf{G}_{s,\lambda}(x)\;&\xrightarrow[]{\;x\to0\;}\;\Lambda_s^{(d)}\;:=\;\frac{\Gamma(\frac{d-s}{2})}{\,(2\pi)^{\frac{d}{2}}\,2^{s-\frac{d}{2}}\,\Gamma(\frac{d}{2})}\,,\qquad\qquad\qquad\qquad s\in(0,d) \\
    \mathsf{G}_{s,\lambda}(x)\;&\xrightarrow[]{\;x\to0\;}\;\mathsf{G}_{s,\lambda}(0)\;=\;\big(2^{d-1}\pi^{\frac{d}{2}-1}\Gamma({\textstyle\frac{d}{2}})\,\lambda^{\frac{s-d}{s}}s\,{\textstyle\sin{\frac{\pi d}{s}}}\big)^{-1},\quad s>d\,.
 \end{split}
\end{equation}
Now, all the considered regimes $s\in(1,2)$ when $d=2$, $s\in(\frac{3}{2},\frac{5}{2})$ when $d=3$, etc.~lie \emph{below} the transition value $s=d$ between the local singular and the local regular behaviour of $\mathsf{G}_{s,\lambda}$, whereas the regime $s\in(\frac{1}{2},\frac{3}{2})$ when $d=1$ lies \emph{across} the transition value $s=1$.

The same type of distinction clearly occurs for the spanning functions \eqref{eq:kerstar-generic}-\eqref{eq:uxifunctions} of $\ker\big((\mathring{\mathsf{k}}^{(s/2)})^*+\lambda\mathbbm{1}\big)$ for higher deficiency index $\mathcal{J}(s,d)$.

In the present context, the peculiarity described above when $d=1$ results in certain different steps of the construction of the self-adjoint extensions of $\mathring{\mathsf{k}}^{(s/2)}$ and ultimately in the type of parametrisation of such extensions, as we shall see.

Therefore, we articulate our discussion on the extensions of $\mathring{\mathsf{k}}^{(s/2)}$ when the deficiency index is one discussing first the three-dimensional case (Section \ref{sec:homog}) and then the one-dimensional case (Section \ref{sec:homog-1}). As commented already, for generic $d\geqslant 2$ the discussion and the final results are completely analogous to $d=3$.

\section{Rank-one singular perturbations of the fractional Laplacian: homogeneous case in dimension three}\label{sec:homog}

In terms of the general discussion of Sec.~\ref{sec:homo-general}, we consider here the operator $\mathring{\mathsf{k}}^{(s/2)}$ on $L^2(\mathbb{R}^3)$ when $s\in(\frac{3}{2},\frac{5}{2})$. $\mathring{\mathsf{k}}^{(s/2)}$ acts as the fractional Laplacian $(-\Delta)^{s/2}$ on the domain
\begin{equation}\label{eq:domkrings}
 \mathcal{D}(\mathring{\mathsf{k}}^{(s/2)})\;=\;\Big\{f\in H^s(\mathbb{R}^3)\,\Big|\int_{\mathbb{R}^3}\widehat{f}(p)\,\ud p=0\Big\} 
\end{equation}
and its deficiency index is 1.

One has the following construction.


\begin{theorem}\label{thm:Ktau} Let $s\in(\frac{3}{2},\frac{5}{2})$ and $\lambda>0$.

\begin{itemize}
 \item[(i)] The self-adjoint extensions in $L^2(\mathbb{R}^3)$ of the operator $\mathring{\mathsf{k}}^{(s/2)}$ form the family $(\mathsf{k}^{(s/2)}_\tau)_{\tau\in\mathbb{R}\cup\{\infty\}}$, where $\mathsf{k}^{(s/2)}_\infty$ is its Friedrichs extension, namely the self-adjoint fractional Laplacian $(-\Delta)^{s/2}$, and all other extensions are given by
  \begin{equation}\label{eq:D_kTau}
  \begin{split}
   \mathcal{D}\big(\mathsf{k}^{(s/2)}_\tau\big)\;:=&\;\left\{g\in L^2(\mathbb{R}^3)\!\left|\!
  \begin{array}{c}
  \widehat{g}(p)=\displaystyle\widehat{f^\lambda}(p)+\frac{\tau\,\xi}{(|p|^s+\lambda)^2}+\frac{\xi}{|p|^s+\lambda} \\
  \xi\in\mathbb{C}\,,\;\; f^\lambda\in H^s(\mathbb{R}^3) \,,\;\; \int_{\mathbb{R}^3}\widehat{f^\lambda}(p)\,\ud p=0
  \end{array}\!\!\!\right.\right\} \\
  =&\;\;\Big\{\,g=F^\lambda+{\textstyle\frac{2\pi s^2\sin(\frac{3\pi}{s})\lambda^{2-\frac{3}{s}}}{\tau(s-3)}}\,F^\lambda(0)\, \mathsf{G}_{s,\lambda}\,\Big|\,F^\lambda\in H^s(\mathbb{R}^3) \Big\}\,,
  \end{split}
 \end{equation}
 and
 \begin{equation}\label{eq:action_kTau}
  \big(\mathsf{k}^{(s/2)}_\tau+\lambda\mathbbm{1}\big)g\;:=\;\mathcal{F}^{-1}\Big((|p|^s+\lambda)\Big(\widehat{f^\lambda}+\frac{\tau\,\xi}{(|p|^s+\lambda)^2}\Big)\Big)\,.
 \end{equation}
 \item[(ii)] Each extension is semi-bounded from below and
  \begin{equation}\label{eq:positiveKring-tau_iff_positve_tau}
 \begin{split}
 \inf\sigma(\mathsf{k}^{(s/2)}_\tau+\lambda\mathbbm{1})\;\geqslant \;0\quad&\Leftrightarrow\quad \tau\;\geqslant\; 0 \\
 \inf\sigma(\mathsf{k}^{(s/2)}_\tau+\lambda\mathbbm{1})\;> \;0\quad&\Leftrightarrow\quad \tau\;>\; 0 \\
 (\mathsf{k}^{(s/2)}_\tau+\lambda\mathbbm{1})\textrm{ is invertible}\quad&\Leftrightarrow\quad \tau\neq 0\,.
 \end{split}
 \end{equation}
 \item[(iii)] For  each $\tau\in\mathbb{R}$ the quadratic form of the extension $\mathsf{k}^{(s/2)}_\tau$ is given by
\begin{eqnarray}
   \mathcal{D}[\mathsf{k}^{(s/2)}_\tau]\;&=&\; H^{\frac{s}{2}}(\mathbb{R}^3)\dotplus \mathrm{span}\{\mathsf{G}_{s,\lambda}\}  \label{eq:Kring-tau_form1}\\
 \qquad \mathsf{k}^{(s/2)}_\tau[F^\lambda+\kappa_\lambda \mathsf{G}_{s,\lambda}]\;&=&\; \||\nabla|^{\frac{s}{2}} F^\lambda\|_{L^2(\mathbb{R}^3)}^2-\lambda\|F^\lambda+\kappa_\lambda \mathsf{G}_{s,\lambda}\|_{L^2(\mathbb{R}^3)}^2 \nonumber \\
 & & \quad+\lambda\|F^\lambda\|_{L^2(\mathbb{R}^3)}^2  +\frac{\tau(s-3)}{2\pi s^2\lambda^{2-\frac{3}{s}}\sin(\frac{3\pi}{s})}|\kappa_\lambda|^2 \label{eq:Kring-tau_form2}
 \end{eqnarray}
 for any $F^\lambda\in H^{s/2}(\mathbb{R}^3)$ and $\kappa_\lambda\in\mathbb{C}$.
  \item[(iv)] For $\tau\neq 0$, one has the resolvent identity
  \begin{equation}\label{eq:Ktau_res}
      (\mathsf{k}^{(s/2)}_\tau+\lambda\mathbbm{1})^{-1}\;=\;((-\Delta)^{s/2}+\lambda\mathbbm{1})^{-1}+\frac{2\pi s^2\sin(\frac{3\pi}{s})\lambda^{2-\frac{3}{s}}}{\tau(s-3)}\, |\mathsf{G}_{s,\lambda}\rangle\langle \mathsf{G}_{s,\lambda}|\,.
  \end{equation}
 \end{itemize}
\end{theorem}

\begin{proof}
 The whole construction is based upon the Kre{\u\i}n-Vi\v{s}ik-Birman self-adjoint extension scheme. Since $\ker\big((\mathring{\mathsf{k}}^{(s/2)})^*+\lambda\mathbbm{1}\big)\;=\;\mathrm{span}\{\mathsf{G}_{s,\lambda}\}$ and the Friedrichs extension of $\mathring{\mathsf{k}}^{(s/2)}+\lambda\mathbbm{1}$ is the Fourier multiplier $(|p|^s+\lambda)$, one has the following formula for the adjoint (see, e.g., \cite[Theorem 2.2]{GMO-KVB2017}):
 \[
  \begin{split}
   \mathcal{D}\big((\mathring{\mathsf{k}}^{(s/2)})^*\big)\;&=\;\left\{g\in L^2(\mathbb{R}^3)\!\left|\!
  \begin{array}{c}
  \widehat{g}(p)=\displaystyle\widehat{f^\lambda}(p)+\frac{\eta}{(|p|^s+\lambda)^2}+\frac{\xi}{|p|^s+\lambda} \\
  \eta,\xi\in\mathbb{C}\,,\;\; f^\lambda\in H^s(\mathbb{R}^3) \,,\;\; \int_{\mathbb{R}^3}\widehat{f^\lambda}(p)\,\ud p=0
  \end{array}\!\!\!\right.\right\} \\
  \big((\mathring{\mathsf{k}}^{(s/2)})^*+\lambda\mathbbm{1}\big)g\;&=\;\mathcal{F}^{-1}\Big((|p|^s+\lambda)\Big(\widehat{f^\lambda}+\frac{\eta}{(|p|^s+\lambda)^2}\Big)\Big)\,.
  \end{split}
 \]
 Each element of the one-parameter family of self-adjoint extensions of $\mathring{\mathsf{k}}^{(s/2)}$ is identified (see, e.g., \cite[Theorem 3.4]{GMO-KVB2017}) by the \emph{Birman self-adjointness condition}
 $\eta=\tau\xi$ for some $\tau\in\mathbb{R}\cup\{\infty\}$. This establishes the first line of \eqref{eq:D_kTau}.
  Setting  $\widehat{F^\lambda}:=\widehat{f^\lambda}+(|p|^s+\lambda)^{-2}\tau\xi$, the boundary condition between $F^\lambda$ and $\xi$ in Fourier transform reads
 \[\tag{*}
  \int_{\mathbb{R}^3}\widehat{F^\lambda}(p)\,\ud p\;=\;\xi\,\textstyle\frac{4\pi^2\tau(s-3)}{\,s^2\lambda^{2-\frac{3}{s}}\sin(\frac{3\pi}{s})}\,.
 \]
 Then, from $F^\lambda(0)=(2\pi)^{-\frac{3}{2}}\!\int_{\mathbb{R}^3}\widehat{F^\lambda}\ud p$, and using \eqref{eq:sfG_generic} with $d=3$, the second line of  \eqref{eq:D_kTau} follows. 
 Since $\mathsf{k}^{(s/2)}_\tau$ is a restriction of $(\mathring{\mathsf{k}}^{(s/2)})^*$, from the above action of the adjoint one deduces \eqref{eq:action_kTau}. This completes the proof of part (i).

 Part (ii) lists standard facts of the Kre{\u\i}n-Vi\v{s}ik-Birman theory -- see \cite[Theorems 3.5 and 5.1]{GMO-KVB2017}.

 The quadratic form is characterised in the extension theory \cite[Theorem 3.6]{GMO-KVB2017} by the formulas $\mathcal{D}[\mathsf{k}^{(s/2)}_\tau]=\mathcal{D}[\mathsf{k}^{(s/2)}_F]\dotplus \ker\big((\mathring{\mathsf{k}}^{(s/2)})^*+\lambda\mathbbm{1}\big)$ (`$F$' stands for Friedrichs), whence \eqref{eq:Kring-tau_form1}, and $(\mathsf{k}^{(s/2)}_\tau+\lambda\mathbbm{1})[F^\lambda+\kappa_\lambda\mathsf{G}_{s,\lambda}]=((-\Delta)^{s/2}+\lambda\mathbbm{1})[F^\lambda]+\tau|\kappa_\lambda|^2\|\mathsf{G}_{s,\lambda}\|_{L^2(\mathbb{R}^3)}^2$, whence \eqref{eq:Kring-tau_form2}. The proof of part (iii) is completed.

  Kre{\u\i}n's resolvent formula for deficiency index 1 \cite[Theorem 6.6]{GMO-KVB2017} prescribes
  \[f
   (\mathsf{k}^{(s/2)}_\tau+\lambda\mathbbm{1})^{-1}\;=\;((-\Delta)^{s/2}+\lambda\mathbbm{1})^{-1}+\beta_{\lambda,\tau}\, |\mathsf{G}_{s,\lambda}\rangle\langle \mathsf{G}_{s,\lambda}|
  \]
  for some scalar $\beta_{\lambda,\tau}$ to be determined, whenever $(\mathsf{k}^{(s/2)}_\tau+\lambda\mathbbm{1})$ is invertible, hence for $\tau\neq 0$. Thus, for a generic $h\in L^2(\mathbb{R}^3)$, the element  $g:=(\mathsf{k}^{(s/2)}_\tau+\lambda\mathbbm{1})^{-1}h\in \mathcal{D}(\mathsf{k}^{(s/2)}_\tau)$ reads, in view of \eqref{eq:D_kTau} and of the resolvent formula above,
  \[
   \begin{split}
    \widehat{g}(p)\;=&\;\widehat{F^\lambda}(p)+\frac{\xi_\lambda}{|p|^s+\lambda}\, \\
    \widehat{F^\lambda}(p)\;:=&\; \frac{\widehat{h}(p)}{|p|^s+\lambda}\,,\qquad \xi_\lambda\;:=\;\frac{\beta_{\lambda,\tau}}{\:(2\pi)^3}\int_{\mathbb{R}^3}\frac{\widehat{h}(q)}{|q|^s+\lambda}\,\ud q\,.
   \end{split}
  \]
  The boundary condition (*) for $F^\lambda$ and $\xi_\lambda$ then implies $1=\beta_{\lambda,\tau}\,\frac{\tau(s-3)}{2\pi s^2\sin(\frac{3\pi}{s})\lambda^{2-\frac{3}{s}}}$, which determines $\beta_{\lambda,\tau}$ and proves \eqref{eq:Ktau_res}, thus completing also the proof of (iv).
\end{proof}

In analogy to what argued in Remark \ref{rem:taulambda}, the $\tau$-parametrisation of the family $(\mathsf{k}^{(s/2)}_\tau)_{\tau\in\mathbb{R}\cup\{\infty\}}$ depends on the initially chosen shift $\lambda>0$, meaning that with a different choice $\lambda'>0$ the same self-adjoint realisation previously identified by $\tau$ with shift $\lambda$ is now selected by a different extension parameter $\tau'$. In certain contexts it is more convenient to switch onto a natural parametrisation that identifies one extension irrespectively of the infinitely many different pairs $(\lambda,\tau)$ attached to it by the parametrisation of Theorem \ref{thm:Ktau}. We shall do it in the next Theorem: observe that indeed, as compared to Theorem \ref{thm:Ktau}, here below $\lambda>0$ is arbitrary.

\bigskip

\begin{theorem}\label{thm:Kalpha} Let $s\in(\frac{3}{2},\frac{5}{2})$.

\begin{itemize}
 \item[(i)] The self-adjoint extensions in $L^2(\mathbb{R}^3)$ of the operator $\mathring{\mathsf{k}}^{(s/2)}$ form the family $(\mathsf{k}^{(s/2)}_\alpha)_{\alpha\in\mathbb{R}\cup\{\infty\}}$, where $\mathsf{k}^{(s/2)}_\infty$ is its Friedrichs extension, namely the self-adjoint fractional Laplacian $(-\Delta)^{s/2}$, and all other extensions are given, for arbitrary $\lambda>0$, by
  \begin{equation}\label{eq:domKwithalpha}
 \begin{split}
  \mathcal{D}(\mathsf{k}^{(s/2)}_\alpha)\;&=\;\left\{\left. g=F^\lambda+{\displaystyle\frac{F^\lambda(0)}{\,\alpha-\frac{\lambda^{\frac{3}{s}-1}}{2\pi s\sin(\frac{3\pi}{s})}}\,}\,\mathsf{G}_{s,\lambda}\,\right|F^\lambda\in H^s(\mathbb{R}^3)\right\} \\
  (\mathsf{k}^{(s/2)}_\alpha+\lambda)\,g\;&=\;((-\Delta)^{s/2}+\lambda)\,F^\lambda\,.
 \end{split}
\end{equation}
  \item[(ii)] For  each $\alpha\in\mathbb{R}$ the quadratic form of the extension $\mathsf{k}^{(s/2)}_\alpha$ is given by
\begin{eqnarray}
   \mathcal{D}[\mathsf{k}^{(s/2)}_\alpha]\;&=&\; H^{\frac{s}{2}}(\mathbb{R}^3)\dotplus \mathrm{span}\{\mathsf{G}_{s,\lambda}\}  \label{eq:Kring-alpha_form1}\\
 \qquad \mathsf{k}^{(s/2)}_\alpha[F^\lambda+\kappa_\lambda \mathsf{G}_{s,\lambda}]\;&=&\; \||\nabla|^{\frac{s}{2}} F^\lambda\|_{L^2(\mathbb{R}^3)}^2-\lambda\|F^\lambda+\kappa_\lambda \mathsf{G}_{s,\lambda}\|_{L^2(\mathbb{R}^3)}^2 \nonumber \\
 & & \!\!\!\!\!\!\!\!\!\!\!\!\!\!\!\!\!\!\!\!\!\!\!\!+\,\lambda\|F^\lambda\|_{L^2(\mathbb{R}^3)}^2  +{\textstyle\big(\alpha-\frac{\lambda^{\frac{3}{s}-1}}{2\pi s\sin(\frac{3\pi}{s})}\big)}|\kappa_\lambda|^2 \label{eq:Kring-alpha_form2}
 \end{eqnarray}
 for arbitrary $\lambda>0$.
  \item[(iii)] The resolvent of $\mathsf{k}^{(s/2)}_\alpha$ is given by
  \begin{equation}\label{eq:Kalpha_res}
   \begin{split}
      (\mathsf{k}^{(s/2)}_\alpha+\lambda\mathbbm{1})^{-1}\;&=\;((-\Delta)^{s/2}+\lambda\mathbbm{1})^{-1} \\
      &\quad +{\textstyle\big(\alpha-\frac{\lambda^{\frac{3}{s}-1}}{2\pi s\sin(\frac{3\pi}{s})}\big)^{\!-1}}\, |\mathsf{G}_{s,\lambda}\rangle\langle \mathsf{G}_{s,\lambda}|
   \end{split}
  \end{equation}
  for arbitrary $\lambda>0$.
 \item[(iv)] Each extension is semi-bounded from below, and
 \begin{equation}\label{eq:spec-kalpha}
  \begin{split}
   \sigma_{\mathrm{ess}}(\mathsf{k}^{(s/2)}_\alpha)\;&=\;\sigma_{\mathrm{ac}}(\mathsf{k}^{(s/2)}_\alpha)\;=  \;[0,+\infty)\,,\qquad \sigma_{\mathrm{sc}}(\mathsf{k}^{(s/2)}_\alpha)\;=\;\emptyset\,, \\
   \sigma_\mathrm{disc}(\mathsf{k}^{(s/2)}_\alpha)\;&=\;
   \begin{cases}
    \quad \emptyset & \textrm{ if } \alpha\geqslant 0 \\
    \{E_\alpha^{(s)}\} & \textrm{ if } \alpha < 0\,,
   \end{cases}
  \end{split}
 \end{equation}
 where the eigenvalue $E_\alpha^{(s)}$ is non-degenerate and is given by
 \begin{equation}\label{eq:KalphanegEV}
 E_\alpha^{(s)}\;=\;-\big(2\pi|\alpha|\,s\,\sin(-{\textstyle\frac{3\pi}{s})}\big)^{\frac{s}{3-s}}\,,
\end{equation}
  the (non-normalised) eigenfunction being $\mathsf{G}_{s,\lambda=|E_\alpha^{(s)}|}$.
 \end{itemize}
\end{theorem}

\begin{proof}
Reasoning as in Remark \ref{rem:taulambda}, we seek for the relation $\tau=\tau(\lambda)$ that ensures that all the pairs $(\lambda,\tau(\lambda))$, with $\lambda>0$, preserve the decomposition \eqref{eq:D_kTau}-\eqref{eq:action_kTau} and thus label the same element of the family of extensions.

For chosen $\lambda$ and $\tau$, a function $g\in\mathcal{D}(\mathsf{k}^{(s/2)}_\tau)$ decomposes uniquely as
\[
 \widehat{g}\;=\;\widehat{F^\lambda}+\frac{\xi}{|p|^s+\lambda}\,,\qquad 
 \begin{array}{c}
  F^\lambda\in H^s(\mathbb{R}^3) \\
  \xi\in\mathbb{C}\;\;\,\quad
 \end{array}
 \,,\qquad\int_{\mathbb{R}^3}\widehat{F^\lambda}\ud p\;=\;\xi\,\textstyle\frac{4\pi^2\tau(s-3)}{\,s^2\lambda^{2-\frac{3}{s}}\sin(\frac{3\pi}{s})}\,.
\]
Let now $\lambda'>0$ and $\tau'\in\mathbb{R}$ be such that for the \emph{same} function $g$ in the domain of the \emph{same} self-adjoint realisation $\mathsf{k}^{(s/2)}_\tau$ one also has
\[
 \widehat{g}\;=\;\widehat{F^{\lambda'}}+\frac{\xi'}{|p|^s+\lambda'}\,,\qquad 
 \begin{array}{c}
  F^{\lambda'}\in H^s(\mathbb{R}^3) \\
  \xi'\in\mathbb{C}\;\;\,\quad
 \end{array}
 \,,\qquad\int_{\mathbb{R}^3}\widehat{F^{\lambda'}}\ud p\;=\;\xi'\,\textstyle\frac{4\pi^2\tau'(s-3)}{\,s^2\lambda'^{2-\frac{3}{s}}\sin(\frac{3\pi}{s})}\,.
\]
The new splitting of $g$ is equivalent to
\[
 \xi'\;=\;\xi\,,\qquad F^{\lambda'}\;=\;F^{\lambda}+\frac{\xi}{|p|^s+\lambda}-\frac{\xi'}{|p|^s+\lambda'}\,,
\]
and the boundary condition for $F^{\lambda'}$ and $\xi'$ is equivalent to
\[\tag{*}
 \xi\,{\textstyle\frac{4\pi^2\tau(s-3)}{\,s^2\lambda^{2-\frac{3}{s}}\sin(\frac{3\pi}{s})}}+\int_{\mathbb{R}^3}\Big(\frac{\xi'}{|p|^s+\lambda}-\frac{\xi'}{|p|^s+\lambda'}\Big)\ud p\;=\;\xi'\,\textstyle\frac{4\pi^2\tau'(s-3)}{\,s^2\lambda'^{2-\frac{3}{s}}\sin(\frac{3\pi}{s})}\,.
\]

Let us analyze the integral in (*). Both summands in the integrand diverge, with two identical divergences that cancel out. Thus, by means of the identity $r^2(r^s+\lambda)^{-1}=r^{2-s}-\lambda r^{2-s}(r^s+\lambda)^{-1}$, one has
\[
 \begin{split}
  \int_{\mathbb{R}^3}\Big(\frac{1}{|p|^s+\lambda}-\frac{1}{|p|^s+\lambda'}&\Big)\ud p\;=\;4\pi\lim_{R\to +\infty}\Big(\int_0^R\frac{r^2}{r^s+\lambda}\,\ud r-\int_0^R\frac{r^2}{r^s+\lambda'}\,\ud r\Big) \\
  &=\;4\pi\lim_{R\to +\infty}\Big( \int_0^R\frac{\lambda' \,\ud r}{r^{s-2}(r^s+\lambda')}-\int_0^R\frac{\lambda\,\ud r}{r^{s-2}(r^s+\lambda)}\Big) \\
  &=\;\frac{4\pi^2}{\lambda^{\,1-\frac{3}{s}}s\sin(\frac{3\pi}{s})}-\frac{4\pi^2}{\lambda'^{\,1-\frac{3}{s}}s\sin(\frac{3\pi}{s})}\,.
 \end{split}
\]

Plugging the result of the above computation into (*) yields
\[
  \frac{\tau(3-s)-s\lambda}{\lambda^{2-\frac{3}{s}}\,}\;=\;\frac{\tau'(3-s)-s\lambda'}{\lambda'^{2-\frac{3}{s}}\,}\,,
\]
which shows, in complete analogy to \eqref{eq:taulambdatauprimelprime} when $s=2$, that all pairs $(\lambda,\tau)$ such that
\[\tag{**}
 -\frac{\tau(3-s)-s\lambda}{\,2\pi s^2\sin(\frac{3\pi}{s})\lambda^{2-\frac{3}{s}}\,}\;=:\;\alpha
\]
indeed label the same extension (the pre-factor $-2\pi s^2\sin(\frac{3\pi}{s})$ having being added for convenience). Thus, $\alpha\in\mathbb{R}\cup\{\infty\}$ defined in (**) is the natural parametrisation we were aiming for (and the Friedrichs case $\tau\to +\infty$ corresponds to $\alpha\to +\infty$).

Upon replacing $\frac{2\pi s^2\sin(\frac{3\pi}{s})\lambda^{2-\frac{3}{s}}}{\tau(s-3)}=\big(\alpha-\frac{1}{2\pi s\sin(\frac{3\pi}{s})\lambda^{1-\frac{3}{s}}}\big)^{\!-1}$ in the formulas of Theorem \ref{thm:Ktau} we deduce at once all formulas of parts (i), (ii), and (iii), together of course with the certainty, proved above, that the decompositions are now $\lambda$-independent.


Since the deficiency index is 1, and hence all extensions are a rank-one perturbation, in the resolvent sense, of the self-adjoint fractional Laplacian, then all extensions have the same essential spectrum $[0,+\infty)$ of the latter, and additionally may have \emph{at most} one negative non-degenerate eigenvalue, in any case all extensions are semi-bounded from below -- all these being general facts of the extension theory, see, e.g., \cite[Theorem 5.9 and Corollary 5.10]{GMO-KVB2017}. This proves, in particular, the first line in \eqref{eq:spec-kalpha}.

The occurrence of a negative eigenvalue $E_\alpha=-\lambda$ of an extension $\mathsf{k}^{(s/2)}_\alpha$, for some $\lambda>0$, can be read out from the resolvent formula \eqref{eq:Kalpha_res} as the pole of $ (\mathsf{k}^{(s/2)}_\alpha+\lambda\mathbbm{1})^{-1}$, that is, imposing
\[
 \textstyle \alpha-\frac{1}{2\pi s\sin(\frac{3\pi}{s})\lambda^{1-\frac{3}{s}}}\;=\;0\,,
\]
i.e.,
\[
 \alpha\;=\;-\lambda^{\frac{3-s}{s}}\big(2\pi s\sin(-{\textstyle\frac{3\pi}{s})}\big)^{-1}\,.
\]
The identity above can be only satisfied by some $\lambda>0$ when $\alpha<0$, because $\sin(-{\textstyle\frac{3\pi}{s})}>0$, in which case
\[
 \lambda\;=\;\big(2\pi|\alpha|\,s\,\sin(-{\textstyle\frac{3\pi}{s})}\big)^{\frac{s}{3-s}}\,.
\]
Alternatively, one can argue from \eqref{eq:D_kTau}-\eqref{eq:action_kTau} that the eigenvalue $-\lambda$ must correspond to the eigenfunction $(\frac{1}{|p|^s+\lambda})^\vee$, that is, an element of the domain with only singular component, and to the parameter $\tau=0$, hence with $f^\lambda\equiv 0$ in the notation therein. Then, setting $\tau=0$ in (**) yields the same condition above on $\alpha$ and $\lambda$. This proves \eqref{eq:KalphanegEV} and the second line in \eqref{eq:spec-kalpha} when $\alpha<0$, and it also qualifies the eigenfunction.

When such a bound state is absent, and therefore when $\alpha\geqslant 0$, for what argued before one has $\sigma(\mathsf{k}^{(s/2)}_\alpha)=\sigma_{\mathrm{ess}}(\mathsf{k}^{(s/2)}_\alpha)=[0,+\infty)$. This proves the second line in \eqref{eq:spec-kalpha} when $\alpha\geqslant 0$, and completes the proof of part (iv).
\end{proof}

Mirroring the observations made in the conclusions of Section \ref{sec:halpha_and_pert}, we see that the elements of the domains of both the operators $\mathfrak{h}^{s/2}_\alpha$ (the fractional power of the singular perturbation of $(-\Delta)$) and $\mathsf{k}^{(s/2)}_\alpha$ (the singular perturbation of the fractional power $(-\Delta)^{s/2}$) split into a regular $H^s$-component plus a singular component; however, in the former case the local singularity is $|x|^{-1}$ for all considered powers $s\in(\frac{1}{2},2]$, whereas in the latter it is the singularity of $\mathsf{G}_{s,\lambda}$, namely $|x|^{-(3-s)}$ for all powers $s\in(\frac{3}{2},\frac{5}{2})$.

In either case, a local boundary condition constrains regular and singular components: working out the asymptotics as $x\to 0$ in \eqref{eq:domwithalpha} and \eqref{eq:domKwithalpha} by means of \eqref{eq:sfGlambda_asympt} we find
\begin{equation}\label{eq:localsing}
 \begin{split}
  g(x)\,&\sim\,(\alpha+(4\pi|x|)^{-1})\;\:\quad\quad\textrm{as }x\to 0\,,\quad g\in\mathcal{D}(\mathfrak{h}^{s/2}_\alpha)\,,\;\;\;s\in({\textstyle\frac{3}{2}},2] \\
  g(x)\,&\sim\,(\alpha+\Lambda_s |x|^{-(3-s)})\quad\,\,\,\,\textrm{as }x\to 0\,,\quad g\in\mathcal{D}(\mathsf{k}^{(s/2)}_\alpha)\,,\:s\in({\textstyle\frac{3}{2}},{\textstyle \frac{5}{2}})\,,
 \end{split}
\end{equation}
where $\Lambda_s$ is defined in \eqref{eq:sfGlambda_asympt}. (Observe that $\Lambda_{s=2}=(4\pi)^{-1}$, consistently.)

Furthermore, whereas  $(\mathfrak{h}_\alpha+\lambda\mathbbm{1})^{-s/2}$ is not a finite rank perturbation of $(-\Delta+\lambda\mathbbm{1})^{-s/2}$, $ (\mathsf{k}^{(s/2)}_\alpha+\lambda\mathbbm{1})^{-1}$ is indeed a rank-one perturbation of $((-\Delta)^{s/2}+\lambda\mathbbm{1})^{-1}$.

\section{Rank-one singular perturbations of the fractional Laplacian: homogeneous case in dimension one}\label{sec:homog-1}

In terms of the general discussion of Sec.~\ref{sec:homo-general}, we consider here now the operator $\mathring{\mathsf{k}}^{(s/2)}$ on $L^2(\mathbb{R})$ when $s\in(\frac{1}{2},\frac{3}{2})\!\setminus\!\{1\}$.

We start with identifying the Friedrichs extension $\mathsf{k}^{(s/2)}_F$ of $\mathring{\mathsf{k}}^{(s/2)}$. Unlike the three-dimensional case, the structure of $\mathsf{k}^{(s/2)}_F$ depends on whether $s< 1$ or $s>1$.

\begin{proposition}\label{prop:Friedrichs1D}
 Let $s\in(\frac{1}{2},1)\cup(1,\frac{3}{2})$. 
 \begin{itemize}
  \item[(i)] The quadratic form of the Friedrichs extension $\mathsf{k}^{(s/2)}_F$ of $\mathring{\mathsf{k}}^{(s/2)}$ is 
  \begin{equation}\label{eq:1D_Friedrichs_form}
   \begin{split}
    \mathcal{D}[\mathsf{k}^{(s/2)}_F]\;&=\;
 \begin{cases}
  H^{s/2}(\mathbb{R}) & \textrm{ if }s\in (\frac{1}{2},1) \\
  H^{s/2}_0(\mathbb{R}\!\setminus\!\{0\}) & \textrm{ if }s\in (1,\frac{3}{2})
 \end{cases} \\
  \mathsf{k}^{(s/2)}_F[f,g]\;&=\;\langle\,|\nabla|^{\frac{s}{2}}f,\nabla|^{\frac{s}{2}} g\rangle\,.
   \end{split}
  \end{equation}
 \item[(ii)] When $s\in(\frac{1}{2},1)$, one has
 \begin{equation}\label{eq:1DFriedrichs_I}
  \begin{split}
   \mathcal{D}(\mathsf{k}^{(s/2)}_F)\;&=\;H^s(\mathbb{R}) \\
   \mathsf{k}^{(s/2)}_F f\;&=\;(-\Delta)^{\frac{s}{2}}f\,.
  \end{split}
 \end{equation}
  \item[(iii)] When $s\in(1,\frac{3}{2})$, for every $\lambda>0$ one has
  \begin{equation}\label{eq:1DFriedrichs_II}
  \begin{split}
   \mathcal{D}(\mathsf{k}^{(s/2)}_F)\;&=\;\left\{f=\phi-\frac{\phi(0)}{\mathsf{G}_{s,\lambda}(0)}\,\mathsf{G}_{s,\lambda}\,\Big|\,\phi\in H^s(\mathbb{R})\right\}\\
   (\mathsf{k}^{(s/2)}_F+\lambda\mathbbm{1}) f\;&=\;((-\Delta)^{\frac{s}{2}}+\lambda)\phi\,.
    \end{split}
 \end{equation}
 In particular, $\mathcal{D}(\mathsf{k}^{(s/2)}_F)\subset H^s(\mathbb{R})\dotplus\mathrm{span}\{ \mathsf{G}_{s,\lambda}\}$. In this regime of $s$, $(\mathsf{k}^{(s/2)}_F+\lambda\mathbbm{1})$ has an everywhere defined and bounded inverse on $L^2(\mathbb{R})$ with
 \begin{equation}\label{eq:1DFriedrichs_resolvent}
 (\mathsf{k}^{(s/2)}_F+\lambda\mathbbm{1})^{-1}\;=\;((-\Delta)^{\frac{s}{2}}+\lambda\mathbbm{1})^{-1}-\frac{1}{G_{s,\lambda}(0)}\,|G_{s,\lambda}\rangle\langle G_{s,\lambda}|\,.
\end{equation}
 \end{itemize}
\end{proposition}

\begin{proof}
 Following the standard form construction of the Friedrichs extension (see, e.g., \cite[Theorem A.2]{GMO-KVB2017}), the Friedrichs form domain is the completion of $\mathcal{D}(\mathring{\mathsf{k}}^{(s/2)})$ with respect to the $H^{\frac{s}{2}}$-norm, and therefore
 \[
 \mathcal{D}[\mathsf{k}^{(s/2)}_F]\;=\;\overline{\:H^s_0(\mathbb{R}\!\setminus\!\{0\})\,}^{\|\,\|_{H^{s/2}(\mathbb{R})}}\;=\;
 \begin{cases}
  H^{s/2}(\mathbb{R}) & \textrm{ if }s\in (\frac{1}{2},1) \\
  H^{s/2}_0(\mathbb{R}\!\setminus\!\{0\}) & \textrm{ if }s\in (1,\frac{3}{2})\,,
  \end{cases}
 \]
last identity being proved precisely as \eqref{eq:d-dim-domain}. Moreover, $\mathsf{k}^{(s/2)}_F[f,g]={\displaystyle\lim_{n\to+\infty}}\langle\,|\nabla|^{\frac{s}{2}}f_n,\nabla|^{\frac{s}{2}} g_n\rangle$ for all the sequences $(f_n)_n,(g_n)_n$ of $H^{\frac{s}{2}}$-approximants of $f$ and $g$ respectively, whence $\mathsf{k}^{(s/2)}_F[f,g]=\langle\,|\nabla|^{\frac{s}{2}}f,\nabla|^{\frac{s}{2}} g\rangle$. This completes the proof of part (i).

The self-adjoint operator associated with the form \eqref{eq:1D_Friedrichs_form} is qualified by the formulas
\[
 \begin{split}
  \mathcal{D}(\mathsf{k}^{(s/2)}_F+\lambda\mathbbm{1})\;&=\;\left\{ f\in \mathcal{D}[\mathsf{k}^{(s/2)}_F]\left|\!
  \begin{array}{c}
   \exists\,u_f\in L^2(\mathbb{R})\,\textrm{such that} \\
   \langle\,|\nabla|^{\frac{s}{2}}f,\nabla|^{\frac{s}{2}} g\rangle+\lambda\langle f,g\rangle=\langle u_f,g\rangle \\
   \forall\,g\in \mathcal{D}[\mathsf{k}^{(s/2)}_F]
  \end{array}\!\!
  \right. \right\} \\
  (\mathsf{k}^{(s/2)}_F+\lambda\mathbbm{1})f\;&=\;u_f
  \end{split}
\]
valid for any $\lambda>0$. 
(By density of $H^{s/2}_0(\mathbb{R}\!\setminus\!\{0\})$ in $L^2(\mathbb{R})$, the above $u_f$ is unique.) In particular, $\mathcal{D}(\mathsf{k}^{(s/2)}_F)$ is independent of $\lambda$, only its internal decomposition is. When $s\in(\frac{1}{2},1)$ the condition identifying $f$ and $u_f$ is clearly equivalent to $f\in H^s(\mathbb{R})$ and $u_f=(-\Delta)^\frac{s}{2}f$, which yields part (ii) of the thesis.

When instead $s\in(1,\frac{3}{2})$ the above condition reads
\[\tag{*}
\begin{split}
 & \int_{\mathbb{R}}\overline{|p|^{\frac{s}{2}}\widehat{f}\,}\,|p|^{\frac{s}{2}}\widehat{g}\,\ud p+\lambda\int_{\mathbb{R}}\overline{\widehat{f}\,}\widehat{g}\,\ud p\;=\;\int_{\mathbb{R}}\overline{\widehat{u_f}}\,\widehat{g}\,\ud p \\
 & \forall g\;\textrm{ such that } \int_{\mathbb{R}}|(|p|^\frac{s}{2}+1)\,\widehat{g}(p)|^2\ud p\:<\:+\infty\;\textrm{ and }\;\int_{\mathbb{R}}\widehat{g}\,\ud p\:=\:0\,,
\end{split}
\]
 that is,
\[
  \int_{\mathbb{R}}\overline{\big((|p|^s+\lambda)\widehat{f}\,-\widehat{u_f}\big)}\,\widehat{g}\,\ud p\;=\;0
 \]
 for all the $g$'s indicated in (*) and for some $u_f\in L^2(\mathbb{R})$. It is easy to see that this is the same as
 \[
  \widehat{f}(p)\;=\;\frac{\widehat{u_f}(p)+c}{|p|^s+\lambda}\,,\qquad \widehat{u_f}(p)=(|p|^s+\lambda) \widehat{f}(p)-c
 \]
 for some $c\in\mathbb{C}$.
Now, the condition for $f$ to vanish at $x=0$ and belong to $H^{s/2}(\mathbb{R})$ is equivalent to
\[
  0\;=\;\int_{\mathbb{R}}\frac{\widehat{u_f}(p)+c}{|p|^s+\lambda}\,\ud p \qquad\textrm{ and }\qquad
  +\infty\;>\;\int_{\mathbb{R}}\Big|(|p|^{\frac{s}{2}}+1)\, \frac{\widehat{u_f}(p)+c}{|p|^s+\lambda}\Big|^2\,\ud p\,.
\]
The finiteness of the second integral above is guaranteed by $s>1$ and $u_f\in L^2(\mathbb{R})$, whereas the vanishing of the first integral is the same as
\[
 \int_{\mathbb{R}}\frac{\widehat{u_f}(p)}{|p|^s+\lambda}\,\ud p\;=\;-c\int_{\mathbb{R}}\frac{\ud p}{|p|^s+\lambda}\;=\;-2\pi c\,\mathsf{G}_{s,\lambda}(0)\,,
\]
i.e.,
\[
 \sqrt{2\pi}\,c\;=\;-\frac{1}{\mathsf{G}_{s,\lambda}(0)}\int_{\mathbb{R}}((-\Delta)^\frac{s}{2}+\lambda)^{-1}u_f\,\ud x\,.
\]
Therefore, as a distributional identity,
\[
 \begin{split}
  u_f\;&=\; ((-\Delta)^{\frac{s}{2}}+\lambda)f-\sqrt{2\pi}\,c\,\delta\;=\;((-\Delta)^{\frac{s}{2}}+\lambda)\Big(f+\frac{\int_{\mathbb{R}}((-\Delta)^\frac{s}{2}+\lambda)^{-1}u_f}{\mathsf{G}_{s,\lambda}(0)}\,\mathsf{G}_{s,\lambda}\Big)\,.
 \end{split}
\]
In turn, the condition $u_f\in L^2(\mathbb{R})$ is equivalent to $\phi:=((-\Delta)^\frac{s}{2}+\lambda)^{-1}u_f\in H^s(\mathbb{R})$: in terms of such $\phi$ the previous identity reads
\[
 f\;=\;\phi-\frac{\,\int_{\mathbb{R}}\phi\,\ud x\,}{\mathsf{G}_{s,\lambda}(0)}\,\mathsf{G}_{s,\lambda}\,,
\]
e the condition $f(0)=0$ reads $\int_{\mathbb{R}}\phi\,\ud x=\phi(0)$. We have thus found
\[
  f\;=\;\phi-\frac{\phi(0)}{\mathsf{G}_{s,\lambda}(0)}\,\mathsf{G}_{s,\lambda}\,,\qquad u_f\;=\; ((-\Delta)^{\frac{s}{2}}+\lambda)\phi\,,
\]
which proves \eqref{eq:1DFriedrichs_II}. The following inversion formula is then a straightforward consequence:
\[
   \begin{split}
   (\mathsf{k}^{(s/2)}_F+\lambda\mathbbm{1})^{-1}u\;=&\;\,\phi_u-\frac{\phi_u(0)}{\mathsf{G}_{s,\lambda}(0)}\,\mathsf{G}_{s,\lambda} \\
   \phi_u\;:=&\;\,((-\Delta)^\frac{s}{2}+\lambda)^{-1}u\,\in H^s(\mathbb{R})\,.
  \end{split}
\]
On the other hand,
\[
 \phi_u(0)\;=\;\frac{1}{\sqrt{2\pi}}\int_{\mathbb{R}}\frac{\widehat{u}(p)}{|p|^s+\lambda}\,\ud p\;=\;\langle G_{s,\lambda},u\rangle\,,
\]
whence \eqref{eq:1DFriedrichs_resolvent} and the conclusion of the proof of part (iii).
\end{proof}

As a special case of \eqref{eq:1DFriedrichs_resolvent} above,
\begin{equation}\label{eq:hsl}
\begin{split}
 \qquad \mathsf{h}_{s,\lambda}\;:=&\;\:\sqrt{2\pi}\,(\mathring{\mathsf{k}}^{(s/2)}_F+\lambda\mathbbm{1})^{-1}\,\mathsf{G}_{s,\lambda} \\
 =&\;\:
 \begin{cases}
   \;\mathcal{F}^{-1}\Big(\displaystyle\frac{1}{\,(|p|^s+\lambda)^2}\Big) & \textrm{ if }s\in({\textstyle\frac{1}{2}},1) \\
   \;\mathcal{F}^{-1}\Big(\displaystyle\frac{1}{\,(|p|^s+\lambda)^2}-\frac{s-1}{\lambda\, s}\,\frac{1}{\,|p|^s+\lambda}\Big) & \textrm{ if }s\in(1,{\textstyle\frac{3}{2}})\,.
 \end{cases}
\end{split}
\end{equation}
Indeed, $\langle \mathsf{G}_{s,\lambda},\mathsf{G}_{s,\lambda}\rangle=\frac{s-1}{\lambda^{2-\frac{1}{s}}s^2\sin\frac{\pi}{s}}$ and (as follows from \eqref{eq:sfGlambda_asympt}) $\mathsf{G}_{s,\lambda}(0)=(\lambda^{1-\frac{1}{s}} s\sin\frac{\pi}{s})^{-1}$, whence $\|G_{s,\lambda}\|_{L^2(\mathbb{R})}^2/\mathsf{G}_{s,\lambda}(0)=\frac{s-1}{\lambda\, s}$.
%
%

We can now establish the following construction.

\begin{theorem}\label{thm:Ktau-1} Let $s\in(\frac{1}{2},1)\cup(1,\frac{3}{2})$ and $\lambda>0$. Set
\begin{equation}\label{eq:omega_s}
\omega(s)\;:=\;\frac{s^2\sin{(\frac{\pi}{s})}}{s-1}\,,\qquad \theta_s\;:=\;
\begin{cases}
\:0 &\textrm{ if }s < 1\\
\:1 &\textrm{ if }s>1\,.
\end{cases}
\end{equation}
%
%
\begin{itemize}
 \item[(i)] The self-adjoint extensions in $L^2(\mathbb{R})$ of the operator $\mathring{\mathsf{k}}^{(s/2)}$ form the family $(\mathsf{k}^{(s/2)}_\tau)_{\tau\in\mathbb{R}\cup\{\infty\}}$, where $\mathsf{k}^{(s/2)}_\infty$ is its Friedrichs extension, already qualified in Proposition \ref{prop:Friedrichs1D}, and all other extensions are given by
\begin{eqnarray}
 \!\!\!\!\!\!\!\!\!\!\!\!\mathcal{D}\big(\mathsf{k}^{(s/2)}_\tau\big)\;&:=&\left\{\,g=F^\lambda+\Big(\frac{\omega(s)\lambda^{2-\frac{1}{s}}}{\tau}-\frac{\theta_s}{\mathsf{G}_{s,\lambda}(0)}\Big)\,F^\lambda(0)\, \mathsf{G}_{s,\lambda}\,\Big|\,F^\lambda\in H^s(\mathbb{R}) \right\} \label{eq:D_kTau-1} \\
  \!\!\!\!\!\!\!\!\!\!\!\!\!\!\!\!\!\!\big(\mathsf{k}^{(s/2)}_\tau+\lambda\mathbbm{1}\big)g\;&:=&((-\Delta)^{\frac{s}{2}}+\lambda)F^\lambda\,. \label{eq:action_kTau-1}
\end{eqnarray}
 \item[(ii)] Each extension is semi-bounded from below and
  \begin{equation}\label{eq:positiveKring-tau_iff_positve_tau-1}
 \begin{split}
 \inf\sigma(\mathsf{k}^{(s/2)}_\tau+\lambda\mathbbm{1})\;\geqslant \;0\quad&\Leftrightarrow\quad \tau\;\geqslant\; 0 \\
 \inf\sigma(\mathsf{k}^{(s/2)}_\tau+\lambda\mathbbm{1})\;> \;0\quad&\Leftrightarrow\quad \tau\;>\; 0 \\
 (\mathsf{k}^{(s/2)}_\tau+\lambda\mathbbm{1})\textrm{ is invertible}\quad&\Leftrightarrow\quad \tau\neq 0\,.
 \end{split}
 \end{equation}
 \item[(iii)] For  each $\tau\in\mathbb{R}$ the quadratic form of the extension $\mathsf{k}^{(s/2)}_\tau$ is given by
\begin{eqnarray}
   \mathcal{D}[\mathsf{k}^{(s/2)}_\tau]\;&=&\; 
   \begin{cases}
     \;H^{s/2}(\mathbb{R})\dotplus \mathrm{span}\{\mathsf{G}_{s,\lambda}\}  & \;\textrm{ if }s\in(\frac{1}{2},1) \\
     \;H^{s/2}_0(\mathbb{R}\!\setminus\!\{0\})\dotplus \mathrm{span}\{\mathsf{G}_{s,\lambda}\}  & \;\textrm{ if }s\in(1,\frac{3}{2})
   \end{cases} \label{eq:Kring-tau_form1-1}\\
 \!\!\!\!\!\!\!\!\!\!\!\!\mathsf{k}^{(s/2)}_\tau[F^\lambda+\kappa_\lambda \mathsf{G}_{s,\lambda}]\;&=&\; \||\nabla|^{\frac{s}{2}} F^\lambda\|_{L^2(\mathbb{R})}^2-\lambda\|F^\lambda+\kappa_\lambda \mathsf{G}_{s,\lambda}\|_{L^2(\mathbb{R})}^2 \nonumber \\
 & & \quad+\lambda\|F^\lambda\|_{L^2(\mathbb{R})}^2  +\frac{\tau}{\omega(s)\lambda^{2-\frac{1}{s}}}|\kappa_\lambda|^2 \label{eq:Kring-tau_form2-1}
 \end{eqnarray}
 for any $F^\lambda\in \mathcal{D}[\mathsf{k}^{(s/2)}_F]$ and $\kappa_\lambda\in\mathbb{C}$.
  \item[(iv)] For $\tau\neq 0$, one has the resolvent identity
  \begin{equation}\label{eq:Ktau_res-1}
      (\mathsf{k}^{(s/2)}_\tau+\lambda\mathbbm{1})^{-1}\;=\;((-\Delta)^{s/2}+\lambda\mathbbm{1})^{-1}+\Big(\frac{\omega(s)\lambda^{2-\frac{1}{s}}}{\tau}-\frac{\theta_s}{\mathsf{G}_{s,\lambda}(0)}\Big)\, |\mathsf{G}_{s,\lambda}\rangle\langle \mathsf{G}_{s,\lambda}|\,.
  \end{equation}
 \end{itemize}
\end{theorem}

\begin{proof}
We proceed along the line of the proof of Theorem \ref{thm:Ktau}, based upon the Kre{\u\i}n-Vi\v{s}ik-Birman self-adjoint extension scheme and Proposition \ref{prop:Friedrichs1D}.

The adjoint of $\mathring{\mathsf{k}}^{(s/2)}$ is qualified by
 \[\tag{I}
  \begin{split}
   \mathcal{D}\big((\mathring{\mathsf{k}}^{(s/2)})^*\big)\;&=\;\left\{g\in L^2(\mathbb{R})\left|
  \begin{array}{c}
  \widehat{g}(p)=\displaystyle\widehat{f^\lambda}(p)+\eta\,\widehat{\mathsf{h}_{s,\lambda}}+\frac{\xi}{|p|^s+\lambda} \\
  \eta,\xi\in\mathbb{C}\,,\;\; f^\lambda\in H^s(\mathbb{R}) \,,\;\; \int_{\mathbb{R}}\widehat{f^\lambda}(p)\,\ud p=0
  \end{array}\!\right.\right\} \\
  \big((\mathring{\mathsf{k}}^{(s/2)})^*+\lambda\mathbbm{1}\big)g\;&=\;
  (\mathsf{k}^{(s/2)}_F+\lambda\mathbbm{1})(f^\lambda+\eta\,\mathsf{h}_{s,\lambda})\,,
  \end{split}
 \]
 where $\mathsf{h}_{s,\lambda}$ is the function \eqref{eq:hsl}, and the self-adjoint restrictions of $(\mathring{\mathsf{k}}^{(s/2)})^*$ are qualified by the self-adjointness condition
 $\eta=\tau\xi$ for some $\tau\in\mathbb{R}\cup\{\infty\}$. For clarity of presentation, let us split the discussion into the two regimes $s< 1$ and $s>1$.
 
 \medskip
 
 \underline{\textbf{First case:}} $s\in(\frac{1}{2},1)$.
 
 \medskip
 
 Let $F^\lambda:=f^\lambda+\tau\,\xi\,\mathsf{h}_{s,\lambda}$. When $f^\lambda$ and $\xi$ run over their possible domains, $F^\lambda$ spans the whole Friedrichs domain $H^s(\mathbb{R})$. Moreover,
  \[\tag{II}
  \int_{\mathbb{R}}\widehat{F^\lambda}(p)\,\ud p\;=\;\xi\,\textstyle\frac{2\pi\tau}{\lambda^{2-\frac{1}{s}}\omega(s)}\,.
 \]
 Thus, the first line in (I) and (II) yield \eqref{eq:D_kTau-1}. Owing to \eqref{eq:1DFriedrichs_I} and to the fact that $\mathsf{k}^{(s/2)}_\tau$ is a restriction of $(\mathring{\mathsf{k}}^{(s/2)})^*$, 
 one deduces \eqref{eq:action_kTau-1} from (I). Thus, part (i) is proved.

 Parts (ii) and (iii) are proved as in Theorem \ref{thm:Ktau}: in particular, $\mathcal{D}[\mathsf{k}^{(s/2)}_\tau]=\mathcal{D}[\mathsf{k}^{(s/2)}_F]\dotplus \ker\big((\mathring{\mathsf{k}}^{(s/2)})^*+\lambda\mathbbm{1}\big)$ and \eqref{eq:1D_Friedrichs_form} 
 imply \eqref{eq:Kring-tau_form1-1}, whereas $(\mathsf{k}^{(s/2)}_\tau+\lambda\mathbbm{1})[F^\lambda+\kappa_\lambda\mathsf{G}_{s,\lambda}]=((-\Delta)^{s/2}+\lambda\mathbbm{1})[F^\lambda]+\tau|\kappa_\lambda|^2\|\mathsf{G}_{s,\lambda}\|_{L^2(\mathbb{R})}^2$ and \eqref{eq:1D_Friedrichs_form}  imply \eqref{eq:Kring-tau_form2-1}.

  Kre{\u\i}n's resolvent formula for deficiency index 1 and \eqref{eq:1DFriedrichs_I} prescribe
  \[
   (\mathsf{k}^{(s/2)}_\tau+\lambda\mathbbm{1})^{-1}\;=\;((-\Delta)^{s/2}+\lambda\mathbbm{1})^{-1}+\beta_{\lambda,\tau}\, |\mathsf{G}_{s,\lambda}\rangle\langle \mathsf{G}_{s,\lambda}|
  \]
  for some scalar $\beta_{\lambda,\tau}$ to be determined, whenever $(\mathsf{k}^{(s/2)}_\tau+\lambda\mathbbm{1})$ is invertible, hence for $\tau\neq 0$. Thus, for a generic $h\in L^2(\mathbb{R})$, the element  $g:=(\mathsf{k}^{(s/2)}_\tau+\lambda\mathbbm{1})^{-1}h\in \mathcal{D}(\mathsf{k}^{(s/2)}_\tau)$ reads, in view of \eqref{eq:D_kTau} and of the resolvent formula above,
  \[
   \begin{split}
    \widehat{g}(p)\;=&\;\widehat{F^\lambda}(p)+\frac{\xi_\lambda}{|p|^s+\lambda}\,,  \qquad\widehat{F^\lambda}(p)\;:=\; \frac{\widehat{h}(p)}{|p|^s+\lambda}\,,\qquad \xi_\lambda\;:=\;\frac{\beta_{\lambda,\tau}}{\:2\pi}\int_{\mathbb{R}}\frac{\widehat{h}(q)}{|q|^s+\lambda}\,\ud q\,.
   \end{split}
  \]
  The boundary condition (I) for $F^\lambda$ and $\xi_\lambda$ then implies $1=\beta_{\lambda,\tau}\,\frac{\tau}{\omega(s)\lambda^{2-\frac{1}{s}}}$, which determines $\beta_{\lambda,\tau}$ and proves \eqref{eq:Ktau_res-1}, thus completing also the proof of (iv).

   \medskip
 
 \underline{\textbf{Second case:}} $s\in(1,\frac{3}{2})$.
 
 \medskip
 
  Let $F^\lambda:=f^\lambda+\tau\,\xi\,\mathsf{h}_{s,\lambda}$. (This is for consistency with the first case, but such $F^\lambda$ is not going to be the same as the $F^\lambda$ of the thesis: functions will be renamed later.)  
  When $f^\lambda$ and $\xi$ run over their possible domains, $F^\lambda$ spans the whole Friedrichs domain \eqref{eq:1DFriedrichs_II}. In particular $F^\lambda(0)=0$, which shows that the boundary condition in $g$ between $F^\lambda$ and $g-F^\lambda$ cannot have the form (II) any longer.
  
  Owing to \eqref{eq:1DFriedrichs_II}, we can re-write
  \[\tag{III}
   \mathcal{D}(\mathsf{k}^{(s/2)}_F)\;\ni\; F^\lambda\;=\;f^\lambda+\tau\,\xi\,\mathsf{h}_{s,\lambda}\;=\;\phi^\lambda_{\tau,\xi}-\frac{\phi^\lambda_{\tau,\xi}(0)}{\mathsf{G}_{s,\lambda}(0)}\,\mathsf{G}_{s,\lambda}\,,
  \]
  $\phi^\lambda_{\tau,\xi}$ running over the whole $H^s(\mathbb{R})$ when $F^\lambda$ runs over the whole $\mathcal{D}(\mathsf{k}^{(s/2)}_F)$. Using \eqref{eq:hsl} this is the same as
  \[
   \widehat{f^\lambda}+\frac{\tau\,\xi}{(|p|^s+\lambda)^2}-\tau\,\xi\:\frac{s-1}{\lambda\,s}\,\frac{1}{|p|^s+\lambda}\;=\; \widehat{\phi^\lambda_{\tau,\xi}}-\frac{\phi^\lambda_{\tau,\xi}(0)}{\sqrt{2\pi}\,\mathsf{G}_{s,\lambda}(0)}\,\frac{1}{|p|^s+\lambda}\,,
  \]
 whence the identification
 \[\tag{IV}
  \widehat{\phi^\lambda_{\tau,\xi}}\;=\; \widehat{f^\lambda}+\frac{\tau\,\xi}{(|p|^s+\lambda)^2}\,,\qquad \frac{\phi^\lambda_{\tau,\xi}(0)}{\sqrt{2\pi}\,\mathsf{G}_{s,\lambda}(0)}\;=\;\tau\,\xi\:\frac{s-1}{\lambda\,s}\,.
 \]
 (In fact, a straightforward computations confirms that assuming the first of (IV), the second follows.)

 From (I) (with $\eta=\tau\xi$) we then see that a generic $g\in\mathcal{D}(\mathsf{k}^{(s/2)}_\tau)$ has the form
 \[\tag{V}
  \begin{split}
   \widehat{g}\;&=\;\widehat{F^\lambda}+\sqrt{2\pi}\,\xi\,\widehat{\mathsf{G}}_{s,\lambda}\;=\;\widehat{\phi^\lambda_{\tau,\xi}}+\Big(\sqrt{2\pi}\,\xi-\frac{\phi^\lambda_{\tau,\xi}(0)}{\mathsf{G}_{s,\lambda}(0)}\Big)\,\widehat{\mathsf{G}}_{s,\lambda} \\
   &=\;\widehat{\phi^\lambda_{\tau,\xi}}+\Big(\frac{1}{\tau}\,\frac{\lambda\,s}{s-1}-1\Big)\frac{\phi^\lambda_{\tau,\xi}(0)}{\mathsf{G}_{s,\lambda}(0)}\,\widehat{\mathsf{G}}_{s,\lambda}\,,
  \end{split}
 \]
where we used (III) in the second step and the second identity of (IV) in the third step. Since 
\[
 \frac{\lambda\,s}{s-1}\,\frac{1}{\mathsf{G}_{s,\lambda}(0)}\;=\;\omega(s)\,\lambda^{2-\frac{1}{s}}\,,
\]
as follows from \eqref{eq:sfGlambda_asympt} and \eqref{eq:omega_s}, then (V) becomes
\[\tag{VI}
 g\;=\;\phi^\lambda_{\tau,\xi}+\Big(\frac{\omega(s)\,\lambda^{2-\frac{1}{s}}}{\tau}-\frac{1}{\mathsf{G}_{s,\lambda}(0)}\Big)\,\phi^\lambda_{\tau,\xi}(0)\,\mathsf{G}_{s,\lambda}\,.
\]
Moreover, 
\[\tag{VII}
 (\mathsf{k}^{(s/2)}_\tau+\lambda\mathbbm{1})g\;=\;(\mathsf{k}^{(s/2)}_F+\lambda\mathbbm{1})F^{\lambda}\;=\;((-\Delta)^{\frac{s}{2}}+\lambda)\phi^\lambda_{\tau,\xi}\,,
\]
having used the second line of (I) in the first identity and \eqref{eq:1DFriedrichs_II} in the second identity.

Renaming $\phi^\lambda_{\tau,\xi}$ into $F^\lambda$, now $F^\lambda$ runs over the whole $H^s(\mathbb{R})$ when $g$ runs over the whole $\mathcal{D}(\mathsf{k}^{(s/2)}_\tau)$: this fact and (VI) then yield \eqref{eq:D_kTau-1}, whereas (VII) yields \eqref{eq:action_kTau-1}. Part (i) is proved.

The proof of parts (ii) and (iii) follows the same scheme as in the case $s\in(\frac{1}{2},1)$: thus, 
 $\mathcal{D}[\mathsf{k}^{(s/2)}_\tau]=\mathcal{D}[\mathsf{k}^{(s/2)}_F]\dotplus \ker\big((\mathring{\mathsf{k}}^{(s/2)})^*+\lambda\mathbbm{1}\big)$ and \eqref{eq:1D_Friedrichs_form} 
 imply \eqref{eq:Kring-tau_form1-1}, whereas $(\mathsf{k}^{(s/2)}_\tau+\lambda\mathbbm{1})[F^\lambda+\kappa_\lambda\mathsf{G}_{s,\lambda}]=((-\Delta)^{s/2}+\lambda\mathbbm{1})[F^\lambda]+\tau|\kappa_\lambda|^2\|\mathsf{G}_{s,\lambda}\|_{L^2(\mathbb{R})}^2$ and \eqref{eq:1D_Friedrichs_form}  imply \eqref{eq:Kring-tau_form2-1}.

 Concerning part (iv), Kre{\u\i}n's resolvent formula and \eqref{eq:1DFriedrichs_resolvent} prescribe for all $\tau\neq 0$
  \[\tag{VIII}
   \begin{split}
      (\mathsf{k}^{(s/2)}_\tau+\lambda\mathbbm{1})^{-1}\;&=\;(\mathsf{k}^{(s/2)}_F+\lambda\mathbbm{1})^{-1}+\beta_{\lambda,\tau}\, |\mathsf{G}_{s,\lambda}\rangle\langle \mathsf{G}_{s,\lambda}| \\
      &=\;((-\Delta)^{\frac{s}{2}}+\lambda\mathbbm{1})^{-1}+\Big(\beta_{\lambda,\tau}-\frac{1}{\mathsf{G}_{s,\lambda}(0)}\Big)\, |\mathsf{G}_{s,\lambda}\rangle\langle \mathsf{G}_{s,\lambda}|
   \end{split}
  \]
  for some scalar $\beta_{\lambda,\tau}$ to be determined.
   Owing to (VIII) and to \eqref{eq:D_kTau-1}, in order for 
  \[
   g\;:=\;(\mathsf{k}^{(s/2)}_\tau+\lambda\mathbbm{1})^{-1}h\;=\;((-\Delta)^{\frac{s}{2}}+\lambda\mathbbm{1})^{-1}h+\Big(\beta_{\lambda,\tau}-\frac{1}{\mathsf{G}_{s,\lambda}(0)}\Big)\langle \mathsf{G}_{s,\lambda},h\rangle\,\mathsf{G}_{s,\lambda}
  \]
 to belong to $\mathcal{D}(\mathsf{k}^{(s/2)}_\tau)$ for a generic $h\in L^2(\mathbb{R})$, keeping into account that $F^\lambda:=((-\Delta)^{\frac{s}{2}}+\lambda\mathbbm{1})^{-1}h$ is a generic function in $H^s(\mathbb{R})$ and that $\langle \mathsf{G}_{s,\lambda},h\rangle=\langle \mathsf{G}_{s,\lambda},((-\Delta)^{\frac{s}{2}}+\lambda\mathbbm{1}) F^\lambda\rangle=F^\lambda(0)$, one must necessarily have $\beta_{\lambda,\tau}=\omega(s)\lambda^{2-\frac{1}{s}}/\tau$. Then (VIII) yields \eqref{eq:Ktau_res-1}.
\end{proof}

Analogously to the change of parametrisation from Theorem \ref{thm:Ktau} to Theorem \ref{thm:Kalpha}, we deduce from Theorem \ref{thm:Ktau-1} the following version.

\begin{theorem}\label{thm:Kalpha1D} Let $s\in(\frac{1}{2},1)\cup(1,\frac{3}{2})$ and
\begin{equation}\label{eq:Theta1D}
\Theta(s,\lambda)\;:=\;
\big( \lambda^{1-\frac1s}s\sin{({\textstyle\frac{\pi}{s}})}\big)^{-1}\,,\qquad \lambda>0\,.
\end{equation}
\begin{itemize}
 \item[(i)] The self-adjoint extensions in $L^2(\mathbb{R})$ of the operator $\mathring{\mathsf{k}}^{(s/2)}$ form the family $(\mathsf{k}^{(s/2)}_\alpha)_{\alpha\in\mathbb{R}\cup\{\infty\}}$, where for arbitrary $\lambda>0$
  \begin{equation}\label{eq:domKwithalpha1D}
 \begin{split}
  \mathcal{D}(\mathsf{k}^{(s/2)}_\alpha)\;&=\;\left\{\!\!
  \begin{array}{l}
   g=F^\lambda+{\displaystyle\frac{F^\lambda(0)}{\,\alpha-\Theta(s,\lambda)}}\,\mathsf{G}_{s,\lambda} \\
   \qquad F^\lambda\in H^s(\mathbb{R})
  \end{array}\!\!\right\} \\
  (\mathsf{k}^{(s/2)}_\alpha+\lambda)\,g\;&=\;((-\Delta)^{s/2}+\lambda)\,F^\lambda\,.
 \end{split}
\end{equation}
The Friedrichs extension $\mathsf{k}^{(s/2)}_F$, already qualified in Proposition \ref{prop:Friedrichs1D}, corresponds to $\alpha=\infty$ when $s\in(\frac{1}{2},1)$ and to $\alpha=0$ when $s\in(1,\frac{3}{2})$. For generic $s$, the extension with $\alpha=\infty$ is the ordinary self-adjoint fractional Laplacian $(-\Delta)^{s/2}$ on $L^2(\mathbb{R})$.
  \item[(ii)] For  each $\alpha\in\mathbb{R}$ the quadratic form of the extension $\mathsf{k}^{(s/2)}_\alpha$ is given by
\begin{eqnarray}
   \mathcal{D}[\mathsf{k}^{(s/2)}_\alpha]\;&=&\;
   \begin{cases}
     \;H^{s/2}(\mathbb{R})\dotplus \mathrm{span}\{\mathsf{G}_{s,\lambda}\}  & \;\textrm{ if }s\in(\frac{1}{2},1) \\
     \;H^{s/2}_0(\mathbb{R}\!\setminus\!\{0\})\dotplus \mathrm{span}\{\mathsf{G}_{s,\lambda}\}  & \;\textrm{ if $s\in(1,\frac{3}{2})$ and $\alpha\neq 0$}
   \end{cases}\label{eq:Kring-alpha_form1_1D}\\
 \!\!\!\!\!\!\!\!\!\! \mathsf{k}^{(s/2)}_\alpha[F^\lambda+\kappa_\lambda \mathsf{G}_{s,\lambda}]\;&=&\;\||\nabla|^{\frac{s}{2}} F^\lambda\|_{L^2(\mathbb{R})}^2-\lambda\|F^\lambda+\kappa_\lambda \mathsf{G}_{s,\lambda}\|_{L^2(\mathbb{R})}^2 \nonumber \\
 & & \quad+\,\lambda\|F^\lambda\|_{L^2(\mathbb{R})}^2  +
 \Big(\frac{\theta_s}{\Theta(s,\lambda)}+\frac{1}{\alpha-\Theta(s,\lambda)}\Big)^{\!-1}
  |\kappa_\lambda|^2 \label{eq:Kring-alpha_form2_1D}
 \end{eqnarray}
 for arbitrary $\lambda>0$.
  \item[(iii)] The resolvent of $\mathsf{k}^{(s/2)}_\alpha$ is given by
  \begin{equation}\label{eq:Kalpha_res_1D}
   \begin{split}
      (\mathsf{k}^{(s/2)}_\alpha+\lambda\mathbbm{1})^{-1}\;&=\;((-\Delta)^{s/2}+\lambda\mathbbm{1})^{-1} \\
      &\qquad +\big(\alpha-\Theta(s,\lambda)\big)^{\!-1}\, |\mathsf{G}_{s,\lambda}\rangle\langle \mathsf{G}_{s,\lambda}|
   \end{split}
  \end{equation}
  for arbitrary $\lambda>0$.
 \item[(iv)] For each $\alpha\in\mathbb{R}$ the extension $\mathsf{k}^{(s/2)}_\alpha$ is semi-bounded from below, and
 \begin{equation}\label{eq:spec-kalpha_1D}
\sigma_{\mathrm{ess}}(\mathsf{k}^{(s/2)}_\alpha)\;=\;\sigma_{\mathrm{ac}}(\mathsf{k}^{(s/2)}_\alpha)\;=  \;[0,+\infty)\,,\qquad \sigma_{\mathrm{sc}}(\mathsf{k}^{(s/2)}_\alpha)\;=\;\emptyset\,, 
\end{equation}
\begin{equation}\label{eq:spec-kalpha_mar1D}
\sigma_\mathrm{disc}(\mathsf{k}^{(s/2)}_\alpha)\;=\;
   \begin{cases}
    \quad \emptyset & \textrm{ if } (s-1)\,\alpha\leqslant 0 \\
    \{-E_\alpha^{(s)}\} & \textrm{ if } (s-1)\,\alpha> 0 
   \end{cases}
\end{equation}
 where the eigenvalue $-E_\alpha^{(s)}$ is non-degenerate and is given by
 \begin{equation}\label{eq:KalphanegEV_1D}
E_\alpha^{(s)}\;=
\;\big(\alpha s\sin({\textstyle\frac{\pi}{s})}\big)^{\frac{s}{1-s}}
\end{equation}
 the (non-normalised) eigenfunction being $\mathsf{G}_{s,\lambda=|E_\alpha^{(s)}|}$.
\end{itemize}
\end{theorem}

\begin{proof}
For any two pairs $(\lambda,\tau)$ and $(\lambda',\tau')$ identifying the \emph{same} self-adjoint realisation $\mathsf{k}^{(s/2)}_\tau$, a function $g\in\mathcal{D}(\mathsf{k}^{(s/2)}_\tau)$ decomposes as
\[
\begin{split}
 g\;&=\;F^\lambda+A(\lambda,\tau)F^\lambda(0)\mathsf{G}_{s,\lambda}\;=\;F^{\lambda'}+A(\lambda',\tau')F^{\lambda'}(0)\mathsf{G}_{s,\lambda'}\\
 &\qquad A(\lambda,\tau)\;:=\;\frac{\omega(s)\lambda^{2-\frac{1}{s}}}{\tau}-\frac{\theta_s}{G_{s,\lambda}(0)}\,,
\end{split}
\]
and the uniqueness of the decomposition implies
\[\tag{I}
 F^{\lambda'}\;=\;F^\lambda+A(\lambda,\tau)F^\lambda(0)\mathsf{G}_{s,\lambda}-A(\lambda',\tau')F^{\lambda'}(0)\mathsf{G}_{s,\lambda'}\,.
\]
In order for $F^{\lambda'}$ to belong to $H^s(\mathbb{R})$, the non-$H^s$ singularities at $x=0$ of $\mathsf{G}_{s,\lambda}$ and $\mathsf{G}_{s,\lambda'}$ must cancel out, that is,
\[\tag{II}
 A(\lambda,\tau)F^\lambda(0)\;=\;A(\lambda',\tau')F^{\lambda'}(0)\,.
\]
Plugging (II) into (I) and evaluating of (I) at $x=0$ yields
\[\tag{III}
 \frac{A(\lambda,\tau)F^{\lambda}(0)}{F^{\lambda'}(0)}\;=\;F^{\lambda'}(0)\;=\;F^{\lambda}(0)+A(\lambda,\tau)F^{\lambda}(0)\:\frac{1}{2\pi}\int_{\mathbb{R}}\Big(\frac{1}{|p|^s+\lambda}-\frac{1}{|p|^s+\lambda'}\Big)\,\ud p\,.
\]
A straightforward analysis of the integral above (exploiting the compensation of singularities when $s\in(\frac{1}{2},1)$) shows that
\[\tag{IV}
\begin{split}
 \frac{1}{2\pi}\int_{\mathbb{R}}\Big(\frac{1}{|p|^s+\lambda}-\frac{1}{|p|^s+\lambda'}\Big)\,\ud p\;&=\;
\frac{1}{s\sin\frac{\pi}{s}}\,\big(\lambda^{\frac{1-s}{s}}-{\lambda'}^{\frac{1-s}{s}}\big)\\
 &=\;\Theta(s,\lambda)-\Theta(s,\lambda')\,.
\end{split}
\]
Combining (III) and (IV) together implies that $(\lambda,\tau)$ and $(\lambda',\tau')$ are linked by the relation
\[
 \frac{1}{A(\lambda,\tau)}+\Theta(s,\lambda)\;=\;\frac{1}{A(\lambda',\tau')}+\Theta(s,\lambda')\;=:\;\alpha\,,
\]
i.e.,
\[\tag{V}
 \alpha-\Theta(s,\lambda)\;=\;\Big(\frac{\omega(s)\lambda^{2-\frac{1}{s}}}{\tau}-\frac{\theta_s}{\Theta(s,\lambda)}\Big)^{\!-1}\,,
\]
which gives the natural extension parameter $\alpha$. It is immediate from (V) that the Friedrichs case $\tau\to +\infty$ corresponds to $\alpha\to +\infty$ when $s\in(\frac{1}{2},1)$ and to $\alpha=0$ when $s\in(1,\frac{3}{2})$.

Upon replacing (V) in the formulas of Theorem \ref{thm:Ktau} we deduce parts (i), (ii), and (iii).
Moreover, arguing as in the analogous point of the proof of Theorem \ref{thm:Kalpha}, formulas \eqref{eq:spec-kalpha_1D} follow, and one also concludes that each $\mathsf{k}^{(s/2)}_\alpha$ may have \emph{at most} one negative non-degenerate eigenvalue $E_\alpha=-\lambda$, $\lambda>0$. 

The occurrence of $E_\alpha$ is read out from the resolvent formula \eqref{eq:Kalpha_res_1D} as the pole of $ (\mathsf{k}^{(s/2)}_\alpha+\lambda\mathbbm{1})^{-1}$, that is, imposing $\alpha-\Theta(s,\lambda)=0$ and hence
\[\tag{VI}
\alpha\;=\big(\lambda^{\frac{s-1}{s}}\,s\,\sin({\textstyle\frac{\pi}{s})}\big)^{-1}\;\,.
\]
When $s<1$, (VI) can be only satisfied by some $\lambda>0$ when $\alpha<0$, in which case
\[
 \lambda\;=\big(\alpha s\sin({\textstyle\frac{\pi}{s})}\big)^{\frac{s}{1-s}}\qquad (s<1\,,\;\alpha<0)\,.
\]
When instead $s>1$, a solution $\lambda>0$ to (VI) exists only when $\alpha>0$, and is given by
\[
 \lambda\;=\big(\alpha s\sin({\textstyle\frac{\pi}{s})}\big)^{\frac{s}{1-s}}\qquad (s>1\,,\;\alpha>0)\,.
\]
Hence we proved also \eqref{eq:spec-kalpha_mar1D}, which completes the proof of part (iv).
\end{proof}

In the regime $s\in(1,\frac{3}{2})$ Theorem \ref{thm:Kalpha1D}(ii) can be re-phrased in the following even more natural formulation, which shows that $\mathsf{k}^{(s/2)}_\alpha$ can be equivalently qualified as a \emph{form perturbation} of $(-\Delta)^{s/2}$.

\begin{proposition}
 Let $s\in(1,\frac{3}{2})$.  The self-adjoint extensions in $L^2(\mathbb{R})$ of $\mathring{\mathsf{k}}^{(s/2)}$ form the family $(\mathsf{k}^{(s/2)}_\alpha)_{\alpha\in\mathbb{R}\cup\{\infty\}}$, where $\alpha=0$ labels the Friedrichs extension given by \eqref{eq:1D_Friedrichs_form}, $\alpha=\infty$ labels the ordinary self-adjoint fractional Laplacian $(-\Delta)^{s/2}$, and for $\alpha\in\mathbb{R}\setminus\{0\}$ one has \begin{equation}\label{eq:1D_forms_132}
  \begin{split}
   \mathcal{D}[\mathsf{k}^{(s/2)}_\alpha]\;&=\;H^{s/2}_0(\mathbb{R}\!\setminus\!\{0\})\dotplus \mathrm{span}\{\mathsf{G}_{s,\lambda}\}\;=\;H^{s/2}(\mathbb{R})  \\
   \mathsf{k}^{(s/2)}_\alpha[g]\;&=\;\||\nabla|^{\frac{s}{2}} g\,\|_{L^2(\mathbb{R})}^2-\frac{1}{\alpha}\,|g(0)|^2
  \end{split}
 \end{equation}
 for every $\lambda>0$.
\end{proposition}

\begin{proof}
In view of Theorem \ref{thm:Kalpha1D}(ii), we only need to prove the second line of \eqref{eq:1D_forms_132} for a generic $g\in \mathcal{D}[\mathsf{k}^{(s/2)}_\alpha]$. We set $\Sigma:=(\frac{1}{\Theta(s,\lambda)}-\frac{1}{\alpha-\Theta(s,\lambda)})^{-1}$ for short and decompose $g=F^\lambda+\kappa_\lambda \mathsf{G}_{s,\lambda}$ for some $F^\lambda\in H^{s/2}_0(\mathbb{R}\!\setminus\!\{0\})$ and $\kappa_\lambda=g(0)/\mathsf{G}_{s,\lambda}(0)$. Applying \eqref{eq:Kring-alpha_form2_1D}, we find
\[
 \begin{split}
  \mathsf{k}^{(s/2)}_\alpha[g]\;&=\;-\lambda\|g\|_{L^2(\mathbb{R})}^2+\|\,|\nabla|^{\frac{s}{2}}(g-\kappa_\lambda\mathsf{G}_{s,\lambda})\|_{L^2(\mathbb{R})}^2+\lambda\|g-\kappa_\lambda\mathsf{G}_{s,\lambda}\|_{L^2(\mathbb{R})}^2+\Sigma|\kappa_\lambda|^2 \\
  &=\;\|\,|\nabla|^{\frac{s}{2}}g\|_{L^2(\mathbb{R})}^2+|\kappa_\lambda|^2 \big( \|\,|\nabla|^{\frac{s}{2}}\mathsf{G}_{s,\lambda}\|_{L^2(\mathbb{R})}^2+\lambda\|\mathsf{G}_{s,\lambda}\|_{L^2(\mathbb{R})}^2+\Sigma\big) \\
  &\qquad -2\,\mathfrak{Re}\,\kappa_\lambda\big(\langle |\nabla|^{\frac{s}{2}}g,|\nabla|^{\frac{s}{2}}\mathsf{G}_{s,\lambda}\rangle+\lambda\langle g,\mathsf{G}_{s,\lambda}\rangle \big)\,.
 \end{split}
\]
Since $ \|\,|\nabla|^{\frac{s}{2}}\mathsf{G}_{s,\lambda}\|_{L^2(\mathbb{R})}^2+\lambda\|\mathsf{G}_{s,\lambda}\|_{L^2(\mathbb{R})}^2=\langle \mathsf{G}_{s,\lambda},((-\Delta)^{\frac{s}{2}}+\lambda)\mathsf{G}_{s,\lambda}\rangle=\mathsf{G}_{s,\lambda}(0)=\Theta(s,\lambda)$ and analogously $\langle |\nabla|^{\frac{s}{2}}g,|\nabla|^{\frac{s}{2}}\mathsf{G}_{s,\lambda}\rangle+\lambda\langle g,\mathsf{G}_{s,\lambda}\rangle =\overline{g(0)}$, then
\[
 \mathsf{k}^{(s/2)}_\alpha[g]\;=\;\|\,|\nabla|^{\frac{s}{2}}g\|_{L^2(\mathbb{R})}^2+\frac{\:|g(0)|^2}{\,\Theta(s,\lambda)^2}\,(\Theta(s,\lambda)+\Sigma)-2\:\frac{\:|g(0)|^2}{\,\Theta(s,\lambda)}\,.
\]
The coefficient of the $|g(0)|^2$-term above amounts to
\[
\frac{1}{\,\Theta(s,\lambda)^2}\Big(\Theta(s,\lambda)+\frac{1}{\frac{1}{\Theta(s,\lambda)}-\frac{1}{\alpha-\Theta(s,\lambda)}}\Big)-\frac{2}{\,\Theta(s,\lambda)}\;=\;-\frac{1}{\alpha}\,,
\]
whence indeed $\mathsf{k}^{(s/2)}_\alpha[g]=\||\nabla|^{\frac{s}{2}} g\,\|_{L^2(\mathbb{R})}^2-\alpha^{-1}|g(0)|^2$.
\end{proof}

\section{Rank-one singular perturbations of the fractional Laplacian: inhomogeneous case}\label{sec:inhomog}

For completeness of presentation, in this Section we work out the \emph{inhomogeneous} version of the operator $\mathsf{k}_\tau^{(s/2)}$ of Section \ref{sec:homog}, for concreteness in dimension $d=3$. That is, instead of constructing a singular perturbation of $(-\Delta)^{s/2}$, we consider the singular perturbation of the fractional power $(-\Delta+\mathbbm{1})^{s/2}$. This is going to be the operator $\mathfrak{d}_\tau^{(s/2)}$ introduced informally in \eqref{eq:symb2}.

The conceptual scheme is the very same as in Sections \ref{sec:homo-general} and \ref{sec:homog}, and only certain explicit computations are modified in an easy way. Thus, we content ourselves to state the main results without proofs.

For chosen $\lambda>0$ and $s\in\mathbb{R}$ we set
\begin{equation}\label{eq:calG}
  \mathcal{G}_{s,\lambda}(x)\;:=\;\frac{1}{\:(2\pi)^{\frac{3}{2}}}\Big(\frac{1}{(p^{2}+\lambda)^{s/2}}\Big)^{\!\vee}(x)\,,\qquad x,p\in\mathbb{R}^3\,,
\end{equation}
whence $(-\Delta+\lambda\mathbbm{1})^{s/2} \,\mathcal{G}_{s,\lambda}=\delta(x)$ distributionally. We also set
\begin{equation}
 \mathring{\mathfrak{d}}^{(s/2)}_\lambda\;:=\;\overline{(-\Delta+\lambda\mathbbm{1})^{s/2}\upharpoonright C^\infty_0(\mathbb{R}^3\!\setminus\!\{0\})}
\end{equation}
as an operator closure with respect to the Hilbert space $L^2(\mathbb{R}^3)$. Thus, in comparison to Section \ref{sec:halpha_and_pert}, $\mathcal{G}_{2,\lambda}=G_\lambda$ and $\mathring{\mathfrak{d}}_\lambda^{(1)}=\mathring{\mathfrak{h}}+\lambda\mathbbm{1}$. Moreover,
\begin{equation}
 \begin{split}
  |x|^{3-s}\,\mathcal{G}_{s,\lambda}(x)\;&\xrightarrow[]{\;x\to0\;}\;\Lambda_s\;=\;\frac{\Gamma(\frac{3-s}{2})}{\,(2\pi)^{\frac{3}{2}}\,2^{s-\frac{3}{2}}\,\Gamma(\frac{s}{2})}\,,\qquad\;\;\,\, s\in(0,3) \\
  \mathcal{G}_{s,\lambda}(x)\;&\xrightarrow[]{\;x\to0\;}\;\mathcal{G}_{s,\lambda}(0)\;=\;\frac{\Gamma(\frac{s-3}{2})}{\,8\pi^{\frac{3}{2}}\,\lambda^{\frac{s-3}{2}}\,\Gamma(\frac{s}{2})},\qquad s>3\,.
 \end{split}
\end{equation}
Reasoning as in \eqref{eq:d-dim-domain} and in Lemma \ref{lem:rule}, we see that when $s\in(\frac{3}{2},\frac{5}{2})$ the deficiency index of $\mathring{\mathfrak{d}}_\lambda^{(s/2)}$ equals $1$.

One has the following construction.


\begin{theorem}\label{thm:Dtau} Let $s\in(\frac{3}{2},\frac{5}{2})$ and $\lambda>0$.

\begin{itemize}
 \item[(i)] The self-adjoint extensions in $L^2(\mathbb{R}^3)$ of the operator $\mathring{\mathfrak{d}}^{(s/2)}_\lambda$ form the family $(\mathfrak{d}^{(s/2)}_{\lambda,\tau})_{\tau\in\mathbb{R}\cup\{\infty\}}$, where $\mathfrak{d}^{(s/2)}_{\lambda,\infty}$ is its Friedrichs extension, namely the self-adjoint fractional shifted Laplacian $(-\Delta+\lambda\mathbbm{1})^{s/2}$, and all other extensions are given by
  \begin{equation}\label{eq:D_dTau}
  \begin{split}
   \mathcal{D}\big(\mathfrak{d}^{(s/2)}_{\lambda,\tau}\big)\;:=&\;
   \left\{g\in L^2(\mathbb{R}^3)\!\left|\!
  \begin{array}{c}
  \widehat{g}(p)=\displaystyle\widehat{f^\lambda}(p)+\frac{\tau\,\xi}{(p^2+\lambda)^s}+\frac{\xi}{(p^2+\lambda)^{s/2}} \\
  \xi\in\mathbb{C}\,,\;\; f^\lambda\in H^s(\mathbb{R}^3) \,,\;\; \int_{\mathbb{R}^3}\widehat{f^\lambda}(p)\,\ud p=0
  \end{array}\!\!\!\right.\right\} \\
 =&\;\;
  \Big\{\,g=F^\lambda+{\textstyle\frac{8\pi^{\frac{3}{2}}\lambda^{s-\frac{3}{2}}\Gamma(s)}{\tau\Gamma(s-\frac{3}{2})}}\,F^\lambda(0)\, \mathcal{G}_{s,\lambda}\,\Big|\,F^\lambda\in H^s(\mathbb{R}^3) \Big\}\,,
  \end{split}
 \end{equation}
 where $\widehat{F^\lambda}=\widehat{f^\lambda}(p)+(p^2+\lambda)^{-s}\tau\,\xi$, and
 \begin{equation}\label{eq:action_dTau}
  \mathfrak{d}^{(s/2)}_{\lambda,\tau} g\;:=\;(-\Delta+\lambda\mathbbm{1})^{s/2} F^\lambda\,.
 \end{equation}
 \item[(ii)] Each extension is semi-bounded from below and
  \begin{equation}\label{eq:positiveDDring-tau_iff_positve_tau}
 \begin{split}
 \inf\sigma(\mathfrak{d}^{(s/2)}_{\lambda,\tau})\;\geqslant \;0\quad&\Leftrightarrow\quad \tau\;\geqslant\; 0 \\
 \inf\sigma(\mathfrak{d}^{(s/2)}_{\lambda,\tau})\;> \;0\quad&\Leftrightarrow\quad \tau\;>\; 0 \\
 \mathfrak{d}^{(s/2)}_{\lambda,\tau}\textrm{ is invertible}\quad&\Leftrightarrow\quad \tau\neq 0\,.
 \end{split}
 \end{equation}
 \item[(iii)]  The eigenvalue zero of the extension $\mathfrak{d}^{(s/2)}_{\lambda,\tau=0}$ is non-degenerate and the (non-normalised) eigenfunction is $\mathcal{G}_{s,\lambda}$. When $\tau<0$ the operator $\mathfrak{d}^{(s/2)}_{\lambda,\tau}$ admits one non-degenerate negative eigenvalue $E_\tau<\tau$.
 \item[(iv)] For  each $\tau\in\mathbb{R}$ the quadratic form of the extension $\mathfrak{d}^{(s/2)}_{\lambda,\tau}$ is given by
\begin{eqnarray}
   \mathcal{D}[\mathfrak{d}^{(s/2)}_{\lambda,\tau}]\;&=&\; H^{\frac{s}{2}}(\mathbb{R}^3)\dotplus \mathrm{span}\{\mathcal{G}_{s,\lambda}\}  \label{eq:DDring-tau_form1}\\
 \qquad \mathfrak{d}^{(s/2)}_{\lambda,\tau}[F^\lambda+\kappa_\lambda \mathcal{G}_{s,\lambda}]\;&=&\; \|(-\Delta+\lambda\mathbbm{1})^{s/4}F^\lambda\|_{L^2(\mathbb{R}^3)}^2 +
  {\textstyle\frac{\tau\Gamma(s-\frac{3}{2})}{8\pi^{\frac{3}{2}}\lambda^{s-\frac{3}{2}}\Gamma(s)}}
   |\kappa_\lambda|^2 \label{eq:DDring-tau_form2}
 \end{eqnarray}
 for any $F^\lambda\in H^{\frac{s}{2}}(\mathbb{R}^3)$ and $\kappa_\lambda\in\mathbb{C}$.
  \item[(v)] For $\tau\neq 0$, one has the resolvent identity
  \begin{equation}\label{eq:DDtau_res}
      (\mathfrak{d}^{(s/2)}_{\lambda,\tau})^{-1}\;=\;(-\Delta+\lambda)^{-s/2}+{\textstyle\frac{8\pi^{\frac{3}{2}}\lambda^{s-\frac{3}{2}}\Gamma(s)}{\tau\Gamma(s-\frac{3}{2})}}\, |\mathcal{G}_{s,\lambda}\rangle\langle \mathcal{G}_{s,\lambda}|\,.
  \end{equation}
 \end{itemize}
\end{theorem}

It is clear that, as opposite to $\mathfrak{h}_\tau^{s/2}$ or $\mathsf{k}_\tau^{(s/2)}$, the shift $\lambda>0$ is inherent the very construction of the operator $\mathfrak{d}^{(s/2)}_{\lambda,\tau}$: in fact, the domain of $\mathring{\mathfrak{d}}^{(s/2)}_{\lambda}$ is independent of $\lambda>0$, but its action is not (the difference $\mathring{\mathfrak{d}}^{(s/2)}_{\lambda}-\mathring{\mathfrak{d}}^{(s/2)}_{\lambda'}$ is a bounded, yet non-trivial operator), thus also the adjoint $(\mathring{\mathfrak{d}}^{(s/2)}_{\lambda})^*$ and its self-adjoint restrictions are $\lambda$-dependent (the adjoints $(\mathring{\mathfrak{h}}^{s/2})^*$ and $(\mathring{\mathsf{k}}^{(s/2)})^*$ are $\lambda$-independent, instead).

Let us also elaborate further on part (iii) of the Theorem. As in the previous Sections, both statements are classical facts in the  Kre{\u\i}n-Vi\v{s}ik-Birman extension scheme. Chosen $\lambda>0$ and $s\in(\frac{3}{2},\frac{5}{2})$, the negative eigenvalue $E_\tau$ of the extension $\mathfrak{d}^{(s/2)}_{\lambda,\tau}$ for $\tau<0$ is obtained as follows. Let $g$ be the corresponding eigenfunction and decompose $\widehat{g}=\widehat{f^\lambda}+(p^2+\lambda)^{-s}\tau+(p^2+\lambda)^{-s/2}$ according to \eqref{eq:D_dTau} (without loss of generality we re-absorbed $\xi$ in $f^\lambda$). Then the condition $\mathfrak{d}^{(s/2)}_{\lambda,\tau}g=E_\tau g$, owing to the property \eqref{eq:action_dTau}, reads
\[
 (p^2+\lambda)^{\frac{s}{2}}\widehat{f^\lambda}+(p^2+\lambda)^{-\frac{s}{2}}\tau\;=\;E_\tau\widehat{f^\lambda}+(p^2+\lambda)^{-s}\tau E_\tau+(p^2+\lambda)^{-\frac{s}{2}} E_\tau\,,
\]
whence
\[
 \widehat{f}(p)\;=\;\frac{1}{\,(p^2+\lambda)^{\frac{s}{2}}-E_\tau}\Big(\frac{\tau E_\tau}{\,(p^2+\lambda)^s}-
 \frac{\tau-E_\tau}{(p^2+\lambda)^{\frac{s}{2}}}\Big)\,.
\]
The fact that $f^\lambda\in H^s(\mathbb{R}^3)$ is then obvious, whereas the condition  $\int_{\mathbb{R}^3}\widehat{f^\lambda}\ud p=0$ selects the value of $E_\tau$ in terms of $\tau$ (and of $\lambda$) -- a numerical example is provided in Figure \ref{fig:variational_plot}.

\begin{figure}
\includegraphics[width=7cm]{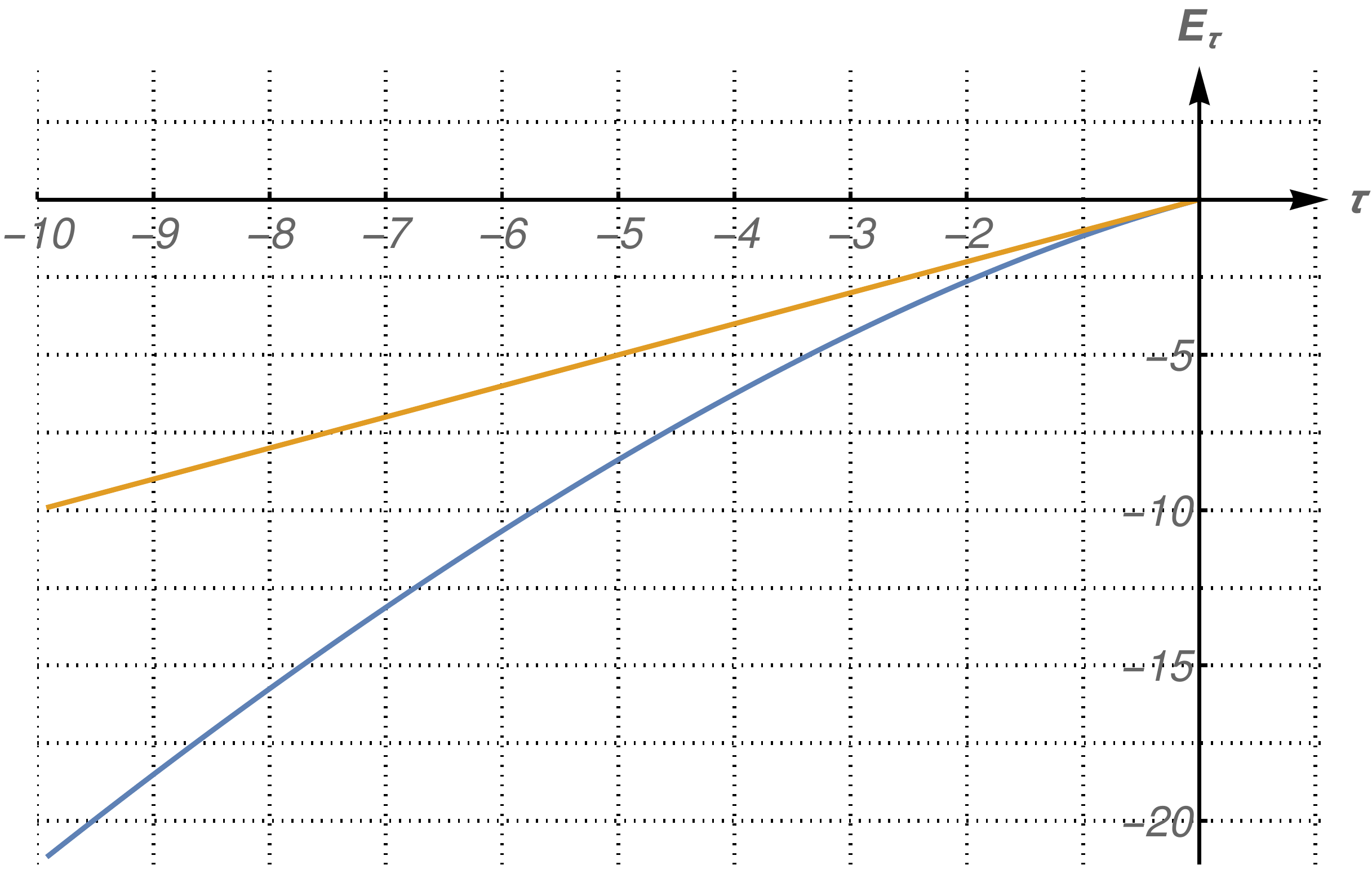}
\caption{Plot of the negative eigenvalue $E_\tau$ of the extension $\mathfrak{d}^{(s/2)}_{\lambda,\tau}$ vs $\tau$ for $\lambda=1$ and $s=1.8$ (blue curve). The reference orange curve gives the corresponding value of $\tau$. Indeed, $E_\tau<\tau$ and $E_{\tau\to 0}=0$.}\label{fig:variational_plot}
\end{figure}

Let us conclude by remarking that two relevant features of the homogeneous case are present in the inhomogeneous case too.

First, and most importantly, the operator $\mathfrak{d}^{(s/2)}_{\lambda,\tau}$ is a rank-one perturbation, in the resolvent sense, of $(-\Delta+\lambda)^{s/2}$, precisely as  $\mathsf{k}^{(s/2)}_\tau$ is a rank-one perturbation of $(-\Delta)^{s/2}$.

Second, the elements of the  domain of $\mathfrak{d}^{(s/2)}_{\lambda,\tau}$ decompose into a regular $H^s$-part and a singular part, constrained to the former by a local boundary condition, where the local singularity when $s\in(\frac{3}{2},\frac{5}{2})$ is of the form $|x|^{3-s}$ as $x\to 0$.

\section{High deficiency index (high fractional power) scenario}\label{sec:high_power}

Let us outline in this Section how the previous constructions of the self-adjoint extensions of the operators $\mathring{\mathsf{k}}^{(s/2)}$ or $\mathring{\mathfrak{d}}^{(s/2)}_{\lambda}$ get modified when $s>\frac{d}{2}+1$.

We recall from Lemma \ref{lem:rule} that when $s\in I_n^{(d)}(\frac{d}{2}+n-1,\frac{d}{2}+n)$, $n\in\mathbb{N}$, one has
\begin{equation}
 \ker\big((\mathring{\mathsf{k}}^{(s/2)})^*+\lambda\mathbbm{1}\big)\;=\;\mathrm{span}\Big\{ u^{\lambda}_{\gamma_1,\dots,\gamma_d}\,\Big|\, \gamma_1,\dots,\gamma_d\in\mathbb{N}_0\,,\;\sum_{j=1}^d\gamma_j\leqslant n-1\Big\}\,,
\end{equation}
and analogously
\begin{equation}
 \ker\big((\mathring{\mathfrak{d}}^{(s/2)}_{\lambda})^*+\lambda\mathbbm{1}\big)\;=\;\mathrm{span}\Big\{ v^{\lambda}_{\gamma_1,\dots,\gamma_d}\,\Big|\, \gamma_1,\dots,\gamma_d\in\mathbb{N}_0\,,\;\sum_{j=1}^d\gamma_j\leqslant n-1\Big\}\,,
\end{equation}
having defined
\begin{equation}
  \widehat{v^{\lambda}}_{\gamma_1,\dots,\gamma_d}(p)\;:=\;\frac{\,p_1^{\gamma_1}\cdots p_d^{\gamma_d}\,}{\,(p^2+\lambda)^{\frac{s}{2}}}\,.
  \end{equation}

  The same extension scheme applied in Section \ref{sec:homog} provides an analogous classification of all the self-adjoint extensions in the case of generic deficiency index $\mathcal{J}(s,d)$, where now \emph{each} extension of $\mathring{\mathsf{k}}^{(s/2)}$ is an operator $\mathsf{k}^{(s/2)}_T$ labelled by a self-adjoint operator $T$ in some subspace $\mathcal{D}(T)$ of $\ker\big((\mathring{\mathsf{k}}^{(s/2)})^*+\lambda\mathbbm{1}\big)\cong\mathbb{C}^{\mathcal{J}(s,d)}$, hence labelled by some $N\times N$ hermitian matrix, $1\leqslant N\leqslant\mathcal{J}(s,d)$.

  Explicitly (see, e.g., \cite[Theorem 3.4]{GMO-KVB2017}),
    \begin{equation}\label{eq:domain_kT}
    \begin{split}
   \mathcal{D}(\mathsf{k}^{(s/2)}_T)\;&=\;\left\{g\in L^2(\mathbb{R}^3)\left|
   \begin{array}{c}
    g=f+(\mathring{\mathsf{k}}^{(s/2)}_F+\lambda\mathbbm{1})^{-1}(Tu+w)+u \\
    \textrm{where }f\in H^s_0(\mathbb{R}^3\!\setminus\!\{0\})\,,\;u\in\mathcal{D}(T)\,, \\
    w\in\ker\big((\mathring{\mathsf{k}}^{(s/2)})^*+\lambda\mathbbm{1}\big)\cap\mathcal{D}(T)^\perp
   \end{array}\!\!
   \right.\right\} \\
    (\mathsf{k}^{(s/2)}_T+\lambda\mathbbm{1})g\;&=\;(\mathsf{k}^{(s/2)}_F+\lambda\mathbbm{1})F_\lambda \\
   &\qquad F_\lambda\;:=\;f+(\mathring{\mathsf{k}}^{(s/2)}_F+\lambda\mathbbm{1})^{-1}(Tu+w)\,\in\, H^s(\mathbb{R}^3)\,,
   \end{split}
  \end{equation}
 with $\mathsf{k}^{(s/2)}_F$ denoting the Friedrichs extension.

 Analogously, the self-adjoint extensions of $\mathring{\mathfrak{d}}^{(s/2)}_{\lambda}$ form a family of operators $\mathfrak{d}^{(s/2)}_{\lambda,T}$ with
     \begin{equation}
      \begin{split}
   \mathcal{D}(\mathfrak{d}^{(s/2)}_{\lambda,T})\;&=\;\left\{g\in L^2(\mathbb{R}^3)\left|
   \begin{array}{c}
    g=f+(\mathring{\mathfrak{d}}^{(s/2)}_{\lambda,F})^{-1}(Tv+w)+v \\
    \textrm{where }f\in H^s_0(\mathbb{R}^3\!\setminus\!\{0\})\,,\;v\in\mathcal{D}(T)\,, \\
    w\in\ker\big(\mathring{\mathfrak{d}}^{(s/2)}_{\lambda})^*\cap\mathcal{D}(T)^\perp
   \end{array}\!\!\!
   \right.\right\} \\
   \mathfrak{d}^{(s/2)}_{\lambda,T} g\;&=\;\mathring{\mathfrak{d}}^{(s/2)}_{\lambda,F}F_\lambda \\
   &\qquad F_\lambda\;:=\;f+(\mathring{\mathfrak{d}}^{(s/2)}_{\lambda,F})^{-1}(Tv+w)\,\in\, H^s(\mathbb{R}^3)\,.
  \end{split}
  \end{equation}

  The theory provides also a counterpart classification of the the quadratic forms of the extensions (see, e.g., \cite[Theorem 3.6]{GMO-KVB2017}). 
  
   The above formulas show that for high powers $s$ the operators $\mathring{\mathsf{k}}^{(s/2)}$ and  $\mathring{\mathfrak{d}}^{(s/2)}_{\lambda}$ have a richer variety (a $\mathcal{J}(s,d)^2$-parameter family) of self-adjoint extensions. The parametrising matrix $T$ determines a more complicated set of `boundary condition' between the `Friedrichs' part of a generic element of the extension domain, and the remaining part: the resulting constraint involves the evaluation at $x=0$ of some number of partial derivatives of the former component.

   This construction produces finite-rank perturbations in the resolvent sense, hence extensions that are all semi-bounded from below and may admit a (finite) number of negative eigenvalues, up to $\mathcal{J}(s,d)$, counting the multiplicity.

 Unlike the case of deficiency index 1, depending on the extension parameter $T$ the large-$p$ vanishing behaviour in momentum space of the singular component may be milder than that of the Green function, and therefore the local singularity of $g$ in position space may be more severe than the behaviour of the Green function as $x\to 0$.

 Let us comment on how the \emph{worst leading singularity} at $x=0$ of a generic function $g\in\mathcal{D}(\mathsf{k}^{(s/2)}_T)$ depends on $s$ and $d$ -- the discussion for $g\in \mathcal{D}(\mathfrak{d}^{(s/2)}_{\lambda,T})$ is identical.

 As expressed by \eqref{eq:domain_kT}, such a singularity is due to those functions of type $u^{\lambda}_{\gamma_1,\dots,\gamma_d}$ that span $\mathcal{D}(T)$. When $s\in I_n^{(d)}$ the worst local singularity occurs when such functions decrease at infinity in momentum coordinates with the slowest possible vanishing rate compatible with $s$ and $d$, that is, when $\gamma_1+\cdots+\gamma_d=n-1$.

 Let $u$ be any such most singular function, which then behaves as $|\widehat{u}(p)|\approx |p|^{-(s+1-n)}$ as $|p|\to +
 \infty$. Then $|u(x)|\approx |x|^{-(d-1+n-s)}$ as $x\to 0$. Since the map 
$$I_n^{(d)}\;\ni\; s\;\mapsto\; d-1+n-s$$
is monotone decreasing and takes values in $(\frac{d}{2}-1,\frac{d}{2})$, if the extension $\mathsf{k}^{(s/2)}_T$ is such that $\mathcal{D}(T)\ni u$, then the functions in $\mathcal{D}(\mathsf{k}^{(s/2)}_T)$ display a local singularity that ranges from $|x|^{-\frac{d}{2}}$ to $|x|^{-\frac{d-1}{2}}$ as long as $s$ increases in $I_n^{(d)}$, precisely as \eqref{eq:localsing} when $s$ increases in $I_1^{(3)}$.

Noticeably, at the transition values $s\in\mathbb{N}+\frac{1}{2}$ the above picture undergoes a discontinuity in $s$, due to the further control of one more derivative in $\mathcal{D}(\mathring{\mathsf{k}}^{(s/2)})$, as a consequence of Sobolev's Lemma, and consequently to emergence in $ \ker\big((\mathring{\mathsf{k}}^{(s/2)})^*+\lambda\mathbbm{1}\big)$ of elements that in momentum coordinates vanish more slowly at infinity.

\section{Applications and perspectives}\label{sec:perspectives}

Besides the operator-theoretic and functional-analytic interest per se of the constructions of the operators \eqref{eq:symb2}, our discussion is profoundly inspired to an amount of natural applications.

Singular perturbations model point-like impurities and more generally point-like interactions. In the realm of the evolutive equations of relevance for quantum mechanics, they naturally govern the evolution of systems subject to such `singular potentials'.

This includes the linear Schr\"{o}dinger evolution 
\[
 \ii\partial_t u\;=\;\mathfrak{h}_\alpha u
\]
as well as the class of semi-linear  Schr\"{o}dinger equations
\[
 \ii\partial_t u\;=\;\mathfrak{h}_\alpha u+\mathcal{N}(u)
\]
of the free Laplacian plus a point-like perturbation,
with physically relevant non-linearities such as the power-law local non-linearity $\mathcal{N}(u)=|u|^{\gamma-1}u$ or the Hartree type non-local non-linearity $\mathcal{N}(u)=(w*|u|^2)u$. The reconstruction of the \emph{linear} propagator from the resolvent of $\mathfrak{h}_\alpha$ is already known in the literature \cite{Scarlatti-Teta-1990,Albeverio_Brzesniak-Dabrowski-1995}, as well as the dispersive properties and space-time estimates of such a propagator \cite{DAncona-Pierfelice-Teta-2006,Iandoli-Scandone-2017}, and the existence, completeness, and $L^p$-boundedness of the wave operators for the pairs $(\mathfrak{h}_\alpha,-\Delta)$ \cite{DMSY-2017}.

In addition, for the study of the \emph{non-linear} problem in a suitable space (the energy space in the first place, as well as other spaces of lower or higher regularity), the knowledge is needed of the corresponding singular norms, namely the norms $ \|u\|_{\widetilde{H}^s_\alpha(\mathbb{R}^3)}=\|(\mathfrak{h}_\alpha+\mathbbm{1})^{s/2}u\|_{L^2(\mathbb{R}^3)}$ considered in Theorem \ref{thm:fracpowGMS} above and \cite{Georgiev-M-Scandone-2016-2017}. In this respect, and with such tools, the study of certain non-linear Schr\"{o}dinger equations with singular potentials has already started \cite{MOS-SingularHartree-2017}.

An analogous \emph{systematic} knowledge for $\mathsf{k}^{(s/2)}_\tau$ and $\mathfrak{d}^{(s/2)}_{\lambda,\tau}$ is by know lacking.

This is even more needed due to the relevance of various models of singular perturbations of fractional differential operators. A relevant example are the \emph{powers} of the  quantum-mechanical semi-relativistic energy operator $\sqrt{-\Delta+m^2}$, the singular perturbation of which yields precisely operators of the type $\mathfrak{d}^{(s/2)}_{m^2,\tau}$ considered in Section \ref{sec:inhomog} or, in the case of zero rest energy, of the type $\mathsf{k}^{(s/2)}_\tau$ as in Section \ref{sec:homog}.

What one finds in the literature is an increasing amount of recent studies \cite{Muslih-IntJTP-2010-3D-fracLaplandDelta,COliveira-Costa-Vaz-JMP2010,COliveira-Vaz-JPA2011_1D_fracLaplandDelta,Lenzi-etAl-JMP2013_fracLapl_andDelta,Sandev-Petreska-Lenzi-JMP2014,Tare-Esguerra-JMP2014,Jarosz-Vaz-JMP2016_1D_gs_fracLaplandDelta,Nayga-Esguerra-JMP-2016} where the singular perturbation of the fractional Laplacian is approached through Green's function methods (together with Wick-like rotations to obtain the propagator from the resolvent) that have the virtue of highlighting the singular structure carried over by what we denoted with $\mathsf{G}_{s,\lambda}$ and $\mathcal{G}_{s,\lambda}$, however with no specific concern to the multiplicity of self-adjoint realisations and the associated local boundary conditions, or to the increase of the deficiency index with the power $s$.

As above, for each extension one would like to qualify the linear propagator, its space-time estimates, the fractional norms, and to use these tools for the associated non-linear problems.

We trust that the research programme emerging from the above considerations may be successfully addressed over the next future!

\appendix

\section{Characterisation of $H^s_0(\mathbb R^d\!\setminus\! \{0\}$)}\label{closurecharact}

We show in this Appendix how to prove the characterisation \eqref{eq:d-dim-domain} of the space $H^s_0(\mathbb R^d\setminus \{0\})$.

It is not restrictive to fix $d=3$ and to discuss and compare the first two regimes $s\in I_0^{(3)}$ and $s\in I_1^{(3)}$. The argument for $s\in I_n^{(3)}$, $n=2,3,\dots$, is completely analogous.

Thus, let us prove the following property.

\begin{lemma}
\begin{equation*}
H^s_0(\mathbb R^3\!\setminus\!\{0\})=\;
 \begin{cases}
  \;\; H^s(\mathbb{R}^3) & \textrm{if } s\in[0,\frac{3}{2}) \\
  \;\; \big\{f\in H^s(\mathbb{R}^3)\,\big|\,\int_{\mathbb{R}^3}\widehat{f}(p)\,\ud p=0\big\} & \textrm{if } s\in(\frac{3}{2},\frac{5}{2})\,.
 \end{cases}
\end{equation*}
\end{lemma}
\begin{proof}
We consider first the case $s\in [0,\frac{3}{2})$. The inclusion 
\begin{align*}
H^s_0(\mathbb R^3\!\setminus\!\{0\})\;\subset\; H^s(\mathbb R^3)
\end{align*} is obvious. For the other inclusion, for any $f\in H^s(\mathbb R^3)$ and for arbitrary $\varepsilon>0$ we want to find $f_{\varepsilon}\in C_0^{\infty}(\mathbb{R}^3\!\setminus\!\{0\})$ such that $\|f_\varepsilon-f\|_{H^s}\leqslant \varepsilon$, and by means of a standard density argument, it is not restrictive to assume further that $f\in \mathcal{S}(\mathbb{R}^3)$ and $\widehat{f}$ is compactly supported. Let $\phi\in C^{\infty}(\mathbb{R}^3)$ be such that
\begin{align*}
\phi(x)= 0 \quad\textrm{for } |x|\leqslant 1 \,, \qquad \phi(x)= 1\quad\textrm{for } |x|\geqslant 2\,,
\end{align*}
and let $\psi:=\phi-1$, $\phi_n(x):=\phi(n|x|)$, and $\psi_n(x):=\psi(n|x|)$, for $n\in\mathbb{N}$. Thus, $\psi\in C^{\infty}_0(\mathbb{R}^3)$ and $\phi_nf\in C^{\infty}(\mathbb{R}^3)$, with $\phi_nf$ vanishing in a neighbourhood of $x=0$. Moreover,
\begin{align*}
\|\phi_nf-f\|_{H^s}\;=\;\|D^s(\phi_n f-f)\|_{L^2}\;=\;\|D^s(\psi_nf)\|_{L^2}\;=\;(2\pi)^{\frac{3}{2}}\|\langle p\rangle^{s}(\widehat\psi_n*\widehat f)\|_{L^2}\,,
\end{align*}
where $D^s:=(1-\Delta)^{s/2}$ and $\langle p\rangle=\sqrt{1+|p|^2}$. Clearly, $\widehat\psi_n(p)=\frac{1}{n^3}\widehat \psi(\frac{p}{n})$. Therefore, 
\begin{align*}
\|\phi_nf-f\|_{H^s}^2\;&=\;(2\pi)^3\int_{\mathbb R^3}\ud p\,\langle p\rangle^{2s}\,\Big| \int_{\mathrm{supp}\,\widehat{f}}\,\ud q\,\widehat f(q)\,\widehat\psi_n(p-q) \Big|^2\\
&\leqslant\; \frac{(2\pi)^3\|\widehat f\|^2_{L^2}}{n^6}\int_{\mathbb R^3}\ud p\int_{\mathrm{supp}\,\widehat{f}}\,\ud q\,\langle p\rangle^{2s}\big| \widehat\psi\big({\textstyle\frac{p-q}{n}}\big)\big|^2\\
&\lesssim\;n^{2s}\int_{\frac{\mathrm{supp}\,\widehat{f}}{n}}\,\ud q\int_{\mathbb R^3}\ud p\,\langle p\rangle^{2s}|\widehat\psi(p-q)|^2
\\
&\lesssim\; n^{2s-3}\;\xrightarrow[]{\;n\to +\infty\;}0\,.
\end{align*}
The last step above follows from the continuity of $q\mapsto \int_{\mathbb R^3}\ud p\,\langle p\rangle^{2s}|\widehat{\psi}(p-q)|^2$. In particular, we can choose $n:=n(\varepsilon)$ sufficiently large such that
$$\|\phi_nf-f\|_{H^s}\;\leqslant\; \frac{\varepsilon}{2}\,.$$
For such $n$, we can find a smooth function $x\mapsto\chi(x)$ that produces a slow cut-off at infinity so that
$$\|\chi\phi_nf-\phi_nf\|_{H^s}\;\leqslant\; \frac{\varepsilon}{2}\,.$$
We have thus identified a function $f_{\varepsilon}:=\chi\phi_nf\in C_0^{\infty}(\mathbb{R}\!\setminus\!\{0\})$ satisfying
$$\|f_{\varepsilon}-f\|_{H^s}\;\leqslant \;\|\chi\phi_nf-\phi_nf\|_{H^s}+\|\phi_nf-f\|_{H^s}\;\leqslant\; \varepsilon\,.$$

Let us discuss now the case $s\in(\frac{3}{2},\frac{5}{2})$. Owing to Sobolev's Lemma, one has the continuous embedding $H^s(\mathbb{R}^3)\hookrightarrow C^{0,s-\frac32}(\mathbb{R}^3)$ and hence any limit in the $H^s$-norm of elements in $C^{\infty}_0(\mathbb{R}^3\!\setminus\!\{0\})$ must vanish at the origin. Therefore, we only need to prove the inclusion
\begin{align*}
H^s_0(\mathbb R^3\!\setminus\!\{0\})\;\supset\;\Big\{f\in H^s(\mathbb R^3)\,\Big|\,\int \widehat f(p)\, \ud p=0 \Big\}\,,
\end{align*}
that is, for any $f\in H^s(\mathbb{R}^3)$ with $f(0)=0$ and for arbitrary $\varepsilon>0$ we want to find $f_\varepsilon\in C^\infty_0(\mathbb{R}^3\!\setminus\!\{0\})$ such that $\|f_\varepsilon-f\|_{H^s}\leqslant \varepsilon$. Since the function $f_R$ defined by
\begin{equation*}
\widehat{f_R}\;:=\;\widehat{f}\cdot\mathbf{1}_{\{|p|\leqslant R\}}-\Big(\int_{|p|\leqslant R}\widehat{f}(p)\ud p\Big)\frac{\mathbf{1}_{\{|p|\leqslant 1\}}}{\big|\{|p|\leqslant 1\}\big|}\,,\qquad R>0\,,
\end{equation*}
has the obvious properties $f_R\in\mathcal{S}(\mathbb{R}^3)$, $\int_{\mathbb{R}^3}\widehat{f_R}(p)\,\ud p=0$, and $\|f-f_R\|_{H^s}\to 0$ as $R\to +\infty$, it is not restrictive to assume from the beginning that $f\in\mathcal{S}(\mathbb{R}^3)$ with $f(0)=0$ and with compactly supported $\widehat{f}$. With the same notation as in the first part of the proof,
\begin{align*}
\|\phi_n f-f\|_{H^s}^2\;&=\;(2\pi)^3\int_{\mathbb R^3}\ud p\,\langle p\rangle^{2s}\Big| \int_{\mathrm{supp}\,\widehat{f}}\,\ud q\,\widehat f(q)\widehat\psi_n(p-q) \Big|^2\\
&=\;(2\pi)^3\int_{\mathbb R^3}\ud p\,\langle p\rangle^{2s}\Big| \int_{\mathrm{supp}\,\widehat{f}}\,\ud q\,\widehat f(q)\big(\widehat\psi_n(p-q)-\widehat \psi_n(p)\big) \Big|^2  \\
&\leqslant\; \frac{(2\pi)^3\|\widehat f\|_{L^2(\mathbb R^3)}^2}{n^6}\int_{\mathbb R^3}\ud p\int_{\mathrm{supp}\,\widehat{f}}\,\ud q\,\langle p\rangle^{2s}\big| \widehat\psi\big({\textstyle\frac{p-q}{n}}\big)-\widehat \psi\big({\textstyle\frac{p}{n}}\big)\big|^2\\
&\lesssim \;n^{2s}\int_{\mathbb R^3}\ud p\int_{\frac{\mathrm{supp}\,\widehat{f}}{n}}\,\ud q\,\langle p\rangle^{2s}\Big| \int_{0}^1\ud t\, (\nabla\widehat \psi)(p-tq)\cdot q  \Big|^2 \\
&\lesssim\; n^{2s-2}\int_{\frac{\mathrm{supp}\,\widehat{f}}{n}}\,\ud q\int_{\mathbb R^3}\ud p\,\langle p\rangle^{2s}\int_0^1\ud t\,|(\nabla\widehat \psi)(p-tq)|^2\\
&\lesssim\; n^{2s-5}\;\xrightarrow[]{\;n\to +\infty\;} 0\,,
\end{align*}
where we used the condition $\int_{\mathbb{R}^3} \widehat f(q)\,\ud q=0$ in the second step, the bound $|q|\lesssim n^{-1}$ for $q\in\frac{\mathrm{supp}\,\widehat{f}}{n}$ in the penultimate step, and the continuity of the function $q\mapsto \int_{\mathbb R^3}\ud p\int_{0}^1\ud t\,\langle p\rangle^{2s}|(\nabla\widehat) \psi(p-tq)|^2$ in the last step. In particular, we can choose $n:=n(\varepsilon)$ sufficiently large such that
$$\|\phi_nf-f\|_{H^s}\;\leqslant\; \frac{\varepsilon}{2}\,.$$
For such $n$, we can find a smooth function $x\mapsto\chi(x)$ that produces a slow cut-off at infinity so that
$$\|\chi\phi_nf-\phi_nf\|_{H^s}\;\leqslant\; \frac{\varepsilon}{2}\,.$$
We have thus identified a function $f_{\varepsilon}:=\chi\phi_nf\in C_0^{\infty}(\mathbb{R}\!\setminus\!\{0\})$ satisfying
$$\|f_{\varepsilon}-f\|_{H^s}\;\leqslant\; \|\chi\phi_nf-\phi_nf\|_{H^s}+\|\phi_nf-f\|_{H^s}\;\leqslant\; \varepsilon,$$
which concludes the proof.
\end{proof}

When $n=2,3,\dots$ and $s\in I_n=(n+\frac{1}{2},n+\frac{3}{2})$, Sobolev's Lemma guarantees that the closure in the $H^s$-norm of $C^{\infty}_0(\mathbb{R}^3\setminus\{0\})$ comes with the vanishing at $x=0$ of the the function and its first partial derivatives up to order $n$. Then one can complete the characterisation of $H^s_0(\mathbb R^3\setminus\{0\})$ by repeating an analogous argument as above, now replacing $f$ with its partial derivatives. This yields the formula
%
%
\begin{equation*}
H^s_0(\mathbb R^3\!\setminus\! \{0\})\;=\;\left\{\!\!
 \begin{array}{c}
  f\in H^s(\mathbb{R}^3) \textrm{ such that}\\
  \int_{\mathbb{R}^3}p_1^{\gamma_1}p_2^{\gamma_2}p_3^{\gamma_3}\widehat{f}(p)\,\ud p=0 \\
  \gamma_1,\gamma_2,\gamma_3\in\mathbb{N}_0\,,\;\gamma_1+\gamma_2+\gamma_3\leqslant n-1
 \end{array}\!\!
 \right\}\,,\qquad s\in \textstyle(n+\frac{1}{2},n+\frac{3}{2})\,.
\end{equation*}

\def\cprime{$'$}

\end{document}